\documentclass{amsart}
\usepackage{amsmath, amssymb, euscript,  enumerate}

\newtheorem{thm}{Theorem}[section]
\newtheorem{cor}[thm]{Corollary}
\newtheorem{lem}[thm]{Lemma}
\newtheorem{prop}[thm]{Proposition}

\newtheorem{defn}[thm]{Definition}
\theoremstyle{remark}
\newtheorem{rem}[thm]{Remark}

\numberwithin{equation}{section}

\theoremstyle{definition}

\newcommand{\R}{\mathbb R}
\newcommand{\Z}{\mathbb Z}
\newcommand{\N}{\mathbb N}

\newcommand{\eps}{\varepsilon}

\newcommand{\e}{\epsilon}

\newcommand{\om}{\Omega}
\newcommand{\al}{\alpha}
\newcommand{\La}{\Lambda}

\newcommand{\la}{\lambda}
\newcommand{\D}{\nabla}
\newcommand{\vphi}{\varphi}
\newcommand{\dist}{\mathrm{dist}}
\newcommand{\pma}{\eqref{eqn-pma}\,}

\newcommand{\ddua}{(\det D^2 u)^p}
\newcommand{\shxt}{S_h(x,t)}

\def\XXint#1#2#3{{\setbox0=\hbox{$#1{#2#3}{\intr}$}
     \vcenter{\hbox{$#2#3$}}\kern-.5\wd0}}

\begin{document}

\title{$C^{1,\al}$ regularity of solutions to parabolic  Monge-Amp\'ere  equations}

\author{Panagiota Daskalopoulos$^*$}

\address{Department of
Mathematics, Columbia University, New York,
 USA}
\email{pdaskalo@math.columbia.edu}

\author{Ovidiu Savin$^{**}$}

\address{Department of
Mathematics, Columbia University, New York,
 USA}
\email{savin@math.columbia.edu}

\thanks{$**:$ Partially supported
by NSF grant 0701037 and a Sloan Fellowship. }

\begin{abstract}
We study interior $C^{1, \al}$ regularity of viscosity solutions
of the
 parabolic Monge-Amp\'ere equation
$$u_t = b(x,t)  \, \ddua,$$
with exponent $p >0$ and with coefficients $b$ which are bounded
and measurable. We show that when $p$ is less than the critical
power $\frac{1}{n-2}$ then solutions become instantly $C^{1, \al}$
in the interior. Also, we prove the same result for any power
$p>0$ at those points where either the solution separates from the
initial data, or where the initial data is $C^{1, \beta}$.
\end{abstract}

\maketitle

\section{Introduction}

In this paper we investigate interior  regularity of viscosity
solutions of the
 parabolic Monge-Amp\'ere equation
\begin{equation}\label{eqn-pma0}
u_t = b(x,t)  \, \ddua,
\end{equation}
with exponent $p >0$ and with coefficients $b$ which are bounded
measurable and satisfy
\begin{equation}\label{eqn-boundsb}
\la \leq b(x,t) \leq \La
\end{equation}
for some fixed constants $\la >0$ and $\La < \infty$. We assume
that the function $u$ is convex in $x$ and increasing in $t$.

Equations of the form of \eqref{eqn-pma0} appear in geometric
evolution problems and in particular in  the motion of a   convex
$n$-dimensional hyper-surface $\Sigma^n_t$ embedded in $\R^{n+1}$
under Gauss curvature flow with exponent $p$, namely the equation
\begin{equation}\label{eqn-gcf1}
\frac {\partial P}{\partial t} = K^p\, {\bf N}
\end{equation}
where  each point $P$ moves in the inward direction ${\bf N}$ to
the surface with velocity equal to the $p$-power of its  Gaussian
curvature  $K$. If we express the surface $\Sigma^n(t)$ locally as
a graph $x_{n+1}=u(x,t)$, with $x \in \om \subset \R^n$, then the
function $u$ satisfies the parabolic Monge-Amp\'ere equation
\begin{equation}\label{eqn-gcf2}
u_t = \frac{(\det D^2 u)^p}{ (1+|\nabla u|^2)^{\frac{(n+2)p-1}2}}.
\end{equation}
Since any convex solution satisfies locally the bound $|\nabla u|
\leq C$, equation \eqref{eqn-gcf2} becomes of the form
\eqref{eqn-pma0}.

The case $p=1$ corresponds to the well studied Gauss curvature
flow which was first introduced by W. Firey in \cite{Fir} as a
model for the wearing process of stones. It follows from the work
of Tso \cite{Tso} that uniformly strictly convex hyper-surfaces
will become instantly $C^\infty$ smooth and they remain smooth up
to their vanishing time $T$. However, convex surfaces which are
not necessarily uniformly strictly convex, may not become
instantly strictly convex and smooth (c.f. \cite{H}, \cite{CEI})
and their regularity poses an interesting problem that we will
investigate in this paper.

Equations of the form \eqref{eqn-gcf1} for different  powers of $p
>0$
 were studied  by B.  Andrews in \cite{An1}
(see also in \cite{Chow}). He showed that when $p \leq 1/n$ any
convex hyper-surface   will become instantly strictly uniformly
convex and smooth.

It can be seen from radially symmetric examples that, when $p
> 1/n$, surfaces evolving by \eqref{eqn-gcf1} or \eqref{eqn-pma0} may have a flat
side that persists for some time before it disappears. These
surfaces are of class $C^{1,\gamma_p}$ with $ \gamma_p :=
\frac{p}{n p-1}$. Since $\gamma_p < 1$ if $p > \frac 1{n-1}$,
solutions which are not strictly convex fail, in general, to be of
class $C^{1,1}$ in this range of exponents. In particular,
solutions to the Gauss curvature flow ($p=1$) with flat sides are
no better than $C^{1,\frac 1{n-1}}$ while the flat sides persist.
The $C^{1, \alpha }$ regularity of solutions of \eqref{eqn-gcf1}
for any $p
>0$ will be addressed in this work.

In dimension $n=2$, the regularity for the Gauss curvature  flow
$(p=1$) is well understood.   It follows from the  work of B.
Andrews in \cite{An2} that,  in this case,  all surfaces  become
instantly of class $C^{1,1}$ and remain so up to a time when they
become strictly convex and therefore smooth, before they contract
to a point.  Also, it follows from the works of the first author
with Hamilton \cite{DH} and Lee \cite{DL} that $C^{1,1}$ is the
optimal regularity here, as can be seen from evolving surfaces
$\Sigma^2_t$ in $\R^3$ with flat sides. The optimal regularity of
surfaces with flat sides and interfaces was further discussed in
\cite{DH, DL}.

We mention that $C^{1,\al}$ and $W^{2,p}$ interior estimates were
established by Guti\'errez and Huang in \cite{GH} for equations
similar to \eqref{eqn-pma0} for $p=-1$ and by Huang and Lu for
$p=\frac{1}{n}$. However, their work requires uniform convexity of
the initial data and strict monotonicity of the function on the
lateral boundary.

\smallskip

If $w$ is a solution to the Monge-Amp\'ere equation
$$\det D^2 w=1, \quad x\in \om \subset \R^n,$$ then $u(x,t)=w(x)+t$
solves equation \eqref{eqn-pma0}  with $b \equiv 1$ for any $p$.
The question of regularity for the Monge-Amp\'ere equation is
closely related to the strict convexity of $w$. Strict convexity
does not always hold in the interior as it can be seen from a
classical example due to Pogorelov \cite{P}. However, Caffarelli
\cite {C1} showed that if the convex set $D$ where $w$ coincides
with a tangent plane contains at least a line segment then all
extremal points of $D$ must lie on $\partial \om$. We prove the
parabolic version of this result for solution of \eqref{eqn-pma0}.
Our result says that, if at a time $t$ the convex set $D$ where
$u$ equals a tangent plane contains at least a line segment then,
either the extremal points of $D$ lie on $\partial \om$ or
$u(\cdot,t)$ coincides with the initial data on $D$ (see Theorem
\ref{thm-fs1}). The second behavior occurs for example in those
solutions with flat sides. In other words, a line segment in the
graph of $u$ at time $t$ either originates from the boundary data
at time $t$ or from the initial data.

We prove a similar result for angles instead of line segments,
which is crucial for our estimates. We show that if at a time $t$
the solution $u$ admits a tangent angle from below then either the
set where $u$ coincides with the edge of the angle has all
extremal points on $\partial \om$ or the initial data has the same
tangent angle from below (see Theorem \ref{thm-angle}).

The $C^{1, \al}$ regularity is closely related to understanding
whether or not solutions separate instantly away from the edges of
a tangent angle of the initial data. It turns out that when $p >
\frac 1{n-2}$ the set where $u$ coincides with the edge of the
angle may persist for some time (see Proposition \ref{prop-p12}),
hence $C^1$ regularity does not hold in this case without further
hypotheses. If $p<\frac 1{n-2}$ we prove that, at any time $t$
after the initial time, solutions are $C^{1,\al}$ in the interior
of any section of $u(\cdot,t)$ which is included in $\om$ (see
Theorem \ref{thm-c1a2}). For the critical exponent
$p=\frac{1}{n-2}$ we show that solutions are $C^1$ with a
logarithmic modulus of continuity for the gradient (see Theorem
\ref{thm-c1a3}).

In the case of any power $p>0$ we prove $C^{1,\al}$ estimates at
all points $(x,t)$ where $u$ separates from the initial data (see
Theorem \ref{thm-c1a4}). Also, if we assume that the initial data
is $C^{1, \beta}$ in some direction $e$ then we show that the
solution is $C^{1,\al}$ in the same direction $e$ for all later
times (see Theorem \ref{thm-c1a1}).

In particular, our methods can be applied for solutions with flat
sides. If the initial data has a flat side $D\subset \R^n$, then
solutions are $C^{1,\al}$ for all later times in the interior of
$D$. A similar statement holds for solutions that contain edges of
tangent angles: they are $C^{1,\al}$ along the direction of the
edge for all later times. To be more precise we state these
results below.

\begin{thm}\label{thm-fb}
Let $u$ be a viscosity solution of \eqref{eqn-pma0} in $\om \times
[0,T]$ with $u(x,0) \ge 0$ in $\om$, $u(x,0) \ge 1$ on $\partial
\om$. There exists $\alpha>0$ depending on $n,\lambda, \Lambda, p$
such that

 a) $u(x,t)$ is $C^{1, \alpha}$ in $x$ at all points
$(x,t)$ with $x$ an interior point of the set $\{u(x,0)=0\}$ and
$u(x,t) <1$.

 b) If $u(x,0) \ge |x_n|$ then $u(x,t)$ is $C^{1,
\alpha}$ in the $x'$ variables at all points $((x',0),t)$ with
$x'$ an interior point of the set $\{x' \, \, : \, \,
u((x',0),0)=0\}$ and $u(x,t)<1$.
\end{thm}

We finally remark that the equations \eqref{eqn-pma0} for negative
and positive powers are in some sense dual to each other. Indeed,
if $u$ is a solution of \eqref{eqn-pma0} and $u^*(\xi,t)$ is the
Legendre transform  of $u(\cdot,t)$ then
$$u^*_t=-\tilde b(\xi,t) (\det D^2 u^*)^{-p}, \quad \quad \la \le \tilde b(\xi,t)\le \La.$$

The paper is organized as follows. In section 2 we introduce the
notation and some geometric properties of sections of convex
functions. In sections 3 we derive estimates for subsolutions and
supersolutions. In section 4 we discuss the  separation of
solutions away from constant solutions such as planes and angles.
In sections 5 and 6 we discuss the geometry of lines and angles.
In section 7 and 8 we quantify the results of section 6 and prove
the main theorems concerning $C^{1, \al}$ regularity.

\smallskip

\section{Preliminaries}

We use the standard notation $B_r(x_0):=\{x \in \R^n \, : \, \,
|x-x_0|<r\}$ to denote the open ball of radius $r$ and center
$x_0$, and we write shortly $B_r$ for $B_r(0)$. Also, given a
point $x=(x_1,\cdots,x_n) \in \R^n$, we denote by $x'$ the point
$x'=(x_1,\cdots,x_{n-1}) \in \R^{n-1}$.

Throughout the paper we refer to positive constants depending on
$n$, $\lambda$, $\Lambda$ and $p$ as universal constants. We
denote them by abuse of notation as $c$ for small constants and
$C$ for large constants, although their values change from line to
line. If a constant depends on universal constants and other
parameters $d$, $\delta$ etc. then we denote them by $c(d,
\delta)$, $C(d,\delta)$.

We use the following definition to say that a function is $C^{1,
\alpha}$ in a pointwise sense.

\begin{defn}\label{def-c1a}
 A function $w$ is $C^{1, \alpha}$ at a point $x_0$ if there
exists a linear function $l(x)$ and a constant $C$ such that
$$|w(x)-l(x)|\le C|x-x_0|^{1+\alpha}.$$

A function is $C^{1, \alpha}$ in a set $D$ if it is $C^{1,
\alpha}$ at each $x \in D$.

A function is $C^{1, \alpha}$ at a point $x_0$ in the direction
$\bf e$ if it is $C^{1, \alpha}$ at $x_0$ when restricted to the
line $x_0+s \bf e$, $s\in \R$.

\end{defn}

 Next we introduce the notion of a section. We denote by $\shxt \subset \R^n$ a section at height $h$ of the
function $u$ at the point $(x,t)$ defined by
$$\shxt := \{ \, y \in \om:  \,\,\,  u(y,t) \leq u(x,t) + p_h \cdot (y-x) + h \, \},$$
for some $p_h \in \R^n$. Sometimes, in order to simplify the
notation, we denote such sections as $S_h$, $S_h(t)$ whenever
there is no possibility of confusion.

We define the notion of $d$-balanced convex set with respect to a
point.

\begin{defn}[$d$-balanced convex set] \label{def-db} A convex set  $S$ with $0 \in S$ is called
 $d$-balanced  with respect to the origin, if there exists  a linear transformation $A$
 (which maps the origin into the origin) such that
$$B_1 \subset A\, S \subset B_d.$$
\end{defn}

Clearly, the notion of $d$-balanced set around $0$ is invariant
under linear transformations. Next we recall

{\bf John's lemma} {\it Every convex set in $\R^n$ is
$C_n$-balanced with respect to its center of mass, with $C_n$
depending only on $n$.}

It is often convenient to consider sections at a point $x$ that
have $x$ as center of mass. We denote such sections by $T_h(x,t)$
instead of $S_h(x,t)$. The existence of centered sections is due
to Caffarelli \cite{C2}.

{\bf Theorem.} [Centered  sections] {\it Let $w$ be a convex
function defined on a bounded convex domain $\om$. For each $x_0
\in \om$, and $h>0$ there exists a centered $h$-section $T_h(x_0)$
at $x_0$
$$T_h(x_0):= \{  w(x) < w(x_0) + h + p_{h} \cdot (x-x_0) \, \}$$
(for some $p_h \in \R^n$)  which has $x_0$ as its center of mass.}

The following simple observations follow from the definition of
$d$-balanced sets and will be used throughout the paper.

\begin{rem}\label{rem-fs1} Assume that the $h$-section of $w$, $$S_h(x_0)=\{w(x) < w(x_0) + h + p_h \cdot (x-x_0)\}, $$ is $d$-balanced around $x_0$.
Then,
$$-d \, h <  w(x) - ( w(x_0) + h + p_h \cdot (x-x_0) ) < 0, \qquad \mbox{in}\,\, \,
S_h(x_0).$$ Also, if we assume $w \ge 0$ and $w(x_0)=0$, then
$w(x) \le d \, h$ for all $x \in S_h(x_0)$.
\end{rem}

Next lemma proves the existence of certain balanced sections which
are compactly included in the domain of definition. It says that
if we have a $d$-balanced section $S_h$ which is compactly
included in $\om$, then we can find $C_nd$-balanced sections for
all smaller heights than $h$ that are included in $S_h$.

\begin{lem}\label{lem-c1a-1} {\em (a)} Assume that $w$ is a convex function
defined on a set $\om \subset \R^n$  with $w(0)=0$ such that
$S_1:= \{ x: \, w(x) < 1 \} \subset \subset \om $ is d-balanced
around $0$. Then, there exists a constant $C_n >0$ depending only
on $n$, such that for every  $h <1$ we can find a section $S_h$ at
height $h$ with  $S_h \subset S_1$ and $S_h$ is $C_nd$-balanced
around $0$.

{\em (b)} Let us denote by $r(x)$ the  volume of the maximal
ellipsoid centered at $x$
 that is included in $S_1$. Then, there exists a number $C_n >0$ such that the section $S_h$
 in part {\em (a)}  is either $C_n$-balanced around  $0$ or $r(x^*) \geq 2\, r(0)$ where
 $x^*$ is the center of mass of $S_h$.
\end{lem}

\begin{proof} a) For $h <1$ fixed, consider the section $S_h$ at height  $h$ that has $0$ as its center of mass. If $S_h \subset S_1$
we have nothing to prove.  Assume not and let's say
$$S_h= \{ w(x) < h + \alpha \, e_n \cdot x \, \}, \qquad \mbox{for some}\,\,\, \alpha >0.$$
We decrease the slope $\alpha$ continuously till we obtain the
section $S_{h,t}:= \{ w < h + t \, e_n \cdot x \, \}$ for which
the set
$$\{ \, (x,x_{n+1}): \,\, x \in \overline S_{h,t}, \,\, x_{n+1} = h + t\, e_n \cdot x \, \}$$
becomes tangent to the hyper-plane $x_{n+1}=1$ at a point
$(x_0,1)$. We will show that $S_{h,t}$ satisfies (a) and (b).

Clearly $S_{h,t}\subset S_1$.  At the point $x_0$ we have $x_0 \in
\partial S_1$ and
$$ S_1 \subset \{ (x-x_0) \cdot e_n \leq 0 \, \}.$$
Since $S_1$ is $d$-balanced, we may assume that $B_1 \subset S_1
\subset B_d$ hence $1\le x_0 \cdot e_n$. Also, $S_h \cap \{ x_n =0
\} = S_{h,t}  \cap \{ x_n =0 \} $, hence the section $S_{h,t}$ is
already $C_n$-balanced in $x':=(x_1,\cdots,x_{n-1})$ around $0$.

Since $t \le \alpha$, the center of mass $x^*$ of $S_{h,t}$
satisfies $x^* \cdot e_n \le 0$. This together with $x_0 \cdot e_n
\ge 1$ and $x_0 \in \partial S_{h,t} \subset \bar B_d$, implies
that $S_{h,t}$ is $C_nd$-balanced around $0$ in all the
directions.

\

b) If we assume that
\begin{equation}\label{eqn-c1a-3}
- x^* \cdot e_n \leq C_0(n) \, x_0 \cdot e_n
\end{equation}
then we obtain that $S_{h,t}$ is $C(n,C_0(n))$ balanced with
respect to $0$. Assume now that \eqref{eqn-c1a-3} doesn't hold and
denote by $E$ the maximum volume ellipsoid centered at $0$ which
is included in $S_1$.  After an affine transformation we have the
following:
$$E=B_1\subset S_1, \quad S_{h,t} \subset \{ \, (x-x_0) \cdot e_n \leq 0 \, \}, \quad x_0 \in \partial S_{h,t}$$
and
$$- x^* \cdot e_n > C_0(n) \, x_0 \cdot e_n \geq C_0(n)$$
which implies that $|x^*| \geq C_0(n)$. Since $x^*$ is the center
of mass of $S_{h,t}$ and $0\in S_{h,t}$ we see from John's lemma
that $(1+ c_n) \, x^* \in S_{h,t} \subset S_1$. Hence if $C_0(n)$
is sufficiently large we can find an ellipsoid of volume $2$
centered at $x^*$ and included in the convex set generated by
$(1+c_n)\, x^*$ and $B_1$. This convex set is contained in $S_1$,
and this concludes the proof of part (b).
\end{proof}

\section{Estimates for subsolutions and supersolutions}

In this section we use the scaling of the equation to derive
estimates for viscosity subsolutions and supersolutions of
\begin{equation}\label{eqn-pma}
\la \, \ddua \leq u_t \leq \La \, \ddua, \quad \quad x \in \om.
\end{equation}
Throughout the paper we assume that $u$ is convex in $x$,
increasing in $t$ and the domain $\om$ is convex and bounded.

Let us now introduce the scaling of equation \eqref{eqn-pma}.
Given an affine transformation $A:= \R^n \to \R^n$ and $ h >0, m
>0$ positive constants, the function
$$v(x,t):= \frac 1h \, u(A\, x,m\, t)$$
is a solution of equation \eqref{eqn-pma} provided
$$m = \frac {(\det A)^{2p}}{h^{np-1}}.$$

The equation is not affected by adding or subtracting a linear
function in $x$. For this reason we write our comparison results
using constant functions instead of linear functions.

\begin{lem}\label{lem-p1} Let  $u$ be a viscosity subsolution in $B_1$ i.e.
$$u_t \le \Lambda (\det D^2 u)^p$$ with $$ \mbox{$u(0,0)\ge -1$, \quad
$u(x,0)\le 0$ on $\partial B_1$.}$$ Then
$$u(0,t) \ge -2 \quad \mbox{for $t \ge -c$},$$
with $c>0$ universal.
\end{lem}

\begin{proof}
If $ u(0,-c) \le -2$ then, by convexity, $u$ at time $-c$ is below
the cone generated by $(0,-2)$ and $\partial B_1$ i.e
$$u(x,-c) \le -2 + 2|x| \quad \mbox{ in $B_1$.}$$ This implies that $u \le w $
on the boundary of the parabolic cylinder $B_1 \times [-c,0]$ for
$$w(x,t):=m(t+c)+2|x|^2-\frac{3}{2} \quad \mbox{with $m=\Lambda 4^{np}$
}.$$ Since $w_t=\Lambda (\det D^2w)^p$ we obtain by the maximum
principle that $u(0,0)\le w(0,0)$ and we reach a contradiction by
choosing $c=1/(4m)$.

\end{proof}

\begin{rem}
The conclusion can be replaced by $u(0,t) \ge -(1+\delta)$ for $t
\ge -c(\delta)$.
\end{rem}

The scaling of the equation and the previous lemma give the
following:

\begin{prop}\label{prop-p2} Assume that $u$ is a viscosity subsolution
in a convex set $S$ with center of mass $0$. If
$$ \mbox{$u(0,0)\ge -h$, \quad $u(x,0)\le 0$ on $\partial S$,}$$
then $$u(0,t) \ge -2h \quad \mbox{for} \quad t \ge -c \,
\frac{|S|^{2p}}{h^{np-1}} ,$$ with $c$ universal.
\end{prop}

\begin{proof} From John's lemma there exists a linear transformation $A$ such
that $$B_1 \subset A^{-1}S \quad \mbox{with} \quad \det A \ge
c(n)|S|.$$ The proposition follows by applying Lemma \ref{lem-p1}
to the rescaled solution
$$v(x,t):= \frac 1h \,  u(A x, m\, t), \qquad m = \frac{(\det A)^{2p}}{h^{np-1}}.$$
\end{proof}

\begin{rem}\label{rem-p2} We obtain a slightly different version of Proposition \ref{prop-p2}
by requiring $S$ to be only $d$-balanced around the origin and by
replacing the conclusion by $u(0,t)\ge -(1+\delta)h$. In this case
we need to take the constant $c=c(d, \delta)$ depending also on
$d$ and $\delta$ as can be seen from the proofs of Lemma
\ref{lem-p1} and Proposition \ref{prop-p2}.
\end{rem}

\begin{rem}\label{rem-p2'}
In general we apply Proposition \ref{prop-p2} at a point
$(x_0,t_0)$ in an $h$-section $S_h=S_h(x_0,t_0)$ which is
$d$-balanced around $x_0$ to conclude that
$$u(x_0,t)\ge u(x_0,t_0)-h \quad \mbox {for} \quad  t \ge t_0-c\frac{|S_h|^{2p}}{h^{np-1}}.$$
\end{rem}

\begin{rem}
At a given point we can apply the Proposition directly in the
sections given by its tangent plane. Indeed, taking $S$ to be the
set
$$S_h=S_h(0,0):=\{u < h+ P(x)\}, \quad P(x):= u(0,0) + \D u(0,0) \cdot x$$
we conclude that $u(x^*,t) \ge P(x^*)-2h$ with $x^*$ the center of
mass of $S_h$. This, by John's lemma, implies a bound in whole
$S_h$
$$u(x,t) \geq P(x) - C(n)\, h, \qquad \mbox{for all} \, \,\, x \in S_h, \,\,\, t \ge - c\, \frac{|S_h|^{2p}}{h^{np-1}}
,$$ with $C(n)$ depending only on the dimension.

\end{rem}

\begin{cor}\label{cor-p3} Assume that $u$ is a bounded subsolution of equation \pma
in the cylinder $Q_1:=B_1 \times [-1,0]$. Then, $u$ is uniformly
H\"older continuous in time $t$  on the cylinder $Q_{1/2}:=B_{1/2}
\times [- 1/2,0]$, namely $u \in C^{1,\beta}(Q_{1/2})$, with
$\beta = 1/(np+1)$.
\end{cor}
\begin{proof}
Since $u$ is bounded on $Q_1$, the convexity of $u(\cdot,t)$
implies that $|\nabla u|$ is bounded by a constant $M$ in
$Q_{3/4}$. Then, by Proposition \ref{prop-p2} applied in $B_h(x)
$, with $x \in B_{1/2}$ and $h<1/4$, we have
\begin{equation}\label{eqn-ht} - 2M \, h \leq u(x,t) - u(x,0)
\leq  0   \qquad \mbox{\em if} \,\,\, - c\,
\frac{|B_h(x)|^{2p}}{h^{np-1}} \leq t \leq 0.
\end{equation}
Taking $t=- c_1\, h^{np +1}$, we find that for all $t$ small
enough
$$|u(x,t) -   u(x,0)|   \leq   C(M)  \,  t^{1/(np+1)} $$
from which the desired result readily follows.
\end{proof}

 As a consequence we obtain compactness of viscosity solutions.

\begin{cor}\label{cor-p4} A sequence of bounded solutions of \pma in $\om \times [-T,0]$
has a subsequence that converges uniformly on compact sets to a
solution of the same equation.
\end{cor}

Next we discuss the case of supersolutions.

\begin{lem}\label{lem-p5} Let $u$ be a viscosity supersolution in $S \subset B_1$ i.e.
$$u_t \geq \la \, (\det D^2 u)^p $$ with $$ \mbox{$u(x,0)\ge -1$ in $S$, \quad \quad $u(x,0)\ge 0$ on $\partial S$.}$$
Then
$$u(x,t) \ge -\frac{1}{2} \quad \mbox{for $t \ge C$},$$
with $C>0$ universal.
\end{lem}

\begin{proof} The lemma follows by  comparison of our solution $u$ with the
function $$w(x,t)= \frac 12 \, (|x|^2 -1) + \la\, (t-C) $$ on the
cylinder $S \times [0,C]$. The function  $w$ is a solution of the
equation $w_t = \la (\det D^2 w)^p$ and, since $S \subset B_1$,
satisfies $w \leq 0$ on $\partial S(0) \times [0,C]$. In addition,
by choosing $C=1/ \la$, we have $w(x,0)\le -1 \le u(x,0)$ for $x
\in S$. The comparison principle implies $u(x,C) \ge w(x,C) \ge
-1/2$ in $S$.

\end{proof}

\begin{rem}
We can replace $-1/2$ by $-\delta$ in the lemma above by taking
$C=C(\delta)$ depending also on $\delta$.
\end{rem}

\begin{rem}{\label{rem-r1}}
If we assume that $S$ is $d$-balanced around $0$ and $u(0,0)=-1$,
$u(x,0)=0$ on $\partial S$, then the same conclusion holds by
taking $C=C(d)$ depending also on $d$. Indeed, in this case we
obtain $u(x,0) \ge -C(d)$ for all $x \in S$ and the desired
conclusion follows as before.
\end{rem}

The scaling of the equation and the previous  lemma give the
following:

\begin{prop}\label{prop-p6} Let $u$ be a supersolution in $\om$
 and assume $$u(x,0)\ge 0, \quad \mbox{ and $S_h :=\{ u(x,0)<h\} \subset \subset \om
 $.}$$
 Then
$$u(\cdot,t) \geq \frac h2, \qquad  \mbox{for} \quad t \ge C \, \frac{|S_h|^{2p}}{h^{np-1}},$$
with $C$ universal.

\end{prop}

\begin{proof} Let $A$ be a linear transformation such that $A^{-1} S_h \subset B_1$ so that
$\det A  \le C(n)|S_h|$. We then apply the previous lemma to the
re-scaled solution
$$v_h=\frac 1h u(Ax,m t)-1, \qquad m = \frac{(\det A)^{2p}}{h^{np-1}}.$$
\end{proof}

\begin{rem}\label{rem-p6}
In view of Remark \ref{rem-r1} we obtain a version of Proposition
\ref{prop-p6} for sections $S_h=S_h(x_0,t_0)$ which are
$d$-balanced around $x_0$ and are compactly included in $\Omega$,
and conclude that
$$u(x_0,t)\ge u(x_0,t_0)+ (1-\delta)h \quad \mbox {for} \quad  t \ge t_0+C(\delta,d)\frac{|S_h|^{2p}}{h^{np-1}}.$$
\end{rem}

\section{Separation from constant solutions}

In this section we consider the case when the solution $u$ at the
initial time $t=0$ is above a given function $w$ depending only on
$n-1$ variables, $u$ and $w$ coincide at the origin, and $u>w$ on
$\partial \om$. We investigate whether $u$ separates from $w$
instantaneously for positive times, i.e $u(0,t)>w(0)$ for all
$t>0$. Of particular interest is the case of angles given by
$w=|x_n|$.

Throughout this section we assume that $u(x,0)\ge 0.$
 For $h >0$ we will consider  the sub-level set $S_h(t)$ of our solution $u(\cdot,t)$
in $\om$ which is defined as
$$S_h(t) := \{ x\in \om: \,\, u(x,t) < h \}.$$
We will also consider balls $B'_\rho \subset \R^{n-1}$, namely
$$B'_\rho := \{ x'=(x_1,\cdots,x_{n-1}) \in  \R^{n-1}: \,\, |x'| < \rho \, \}.$$

\begin{prop}\label{prop-p8}
Let  $u$ be  a supersolution in $\om \times [0,T]$ with $u \geq 0$
at $t=0$. Assume that $S_h(0) \cap \{x_n < 2\, \beta \}$ is
compactly included in $\om$ and is included as well in the
cylinder $ \{0<x_n <2 \beta\} \times S'$ for a bounded domain $S'
\subset \R^{n-1}$ and two positive constants $h
>0, \beta
>0$. Then,
$$S_{h/4}(t_0) \subset \{x_n > \beta  \},
\qquad \mbox{for}\, \,\, t_0 = C \, \frac{(\beta \,
|S'|)^{2p}}{h^{np-1}},$$ with $C$ universal.

\end{prop}

\begin{proof}
We apply Proposition \ref{prop-p2} for $$\tilde u=u+\frac{h}{2
\beta}x_n$$ and see that $\tilde u \ge u \ge 0$. Also $\{\tilde
u(x,0)< h\}$ is compactly included in $\om$ and is included in $
\{0<x_n <2 \beta\} \times S'$. We conclude that $\tilde u(x,t_0)
\ge \frac{3}{4}h$ with $t_0$ given above. This implies that if
$x_n \le \beta$ then $u(x,t_0) \ge \tilde u(x,t_0) -\frac{h}{2}
\ge \frac h4$ hence $S_{h/4}(t_0) \subset \{x_n > \beta \}$.
\end{proof}

From Proposition \ref{prop-p8} we obtain the following corollary.

\begin{cor}\label{cor-p9} Let  $u$ be  a supersolution in $\om \times [0,T]$
and assume that $$u(x,0) \geq w(x') \geq 0,\quad \quad
u(0,0)=w(0)=0, \quad u(x,0)>0 \quad \mbox{on $\partial \om$,}$$
for a function $w$ defined on $\R^{n-1}$. Suppose that $w$
satisfies
\begin{equation}\label{eqn-limit}
\frac{a_{h_j}^{2p}}{h_j^{np-1}} \to 0, \qquad \mbox{for a
sequence}\,\, h_j \to 0.
\end{equation}
with $$a_{h}:=|\{ w(x')<h\}\cap \pi_n(\Omega)|, \quad \quad
\mbox{where $\pi_n(x):=x'.$} $$ Then,
$$u(0,t) > 0 \qquad \mbox{for any}\,\, t >0.$$
\end{cor}

\begin{proof}
Let $h >0$ be small such that the sub-level sets $S_h(0)$ of $u$
is compactly supported in $\om$. Since $u \ge w$ we obtain that
$$S_h(0) \subset (\{ w(x')<h\}\cap \pi_n(\om)) \times [b, \infty),$$ for some $b <0$ (since $0 \in S_h(0)$). We apply
Proposition \ref{prop-p8} for $h_j \le h$ (hence $S_{h_j}(0)
\subset S_h(0)$), with $\beta=-b$. We conclude that
$$S_{h_j/4}(t_j) \subset \{x_n>0\},
\quad \quad t_j=C\beta^{2p}\frac{a_{h_j}^{2p}}{h_j^{np-1}},$$ and
obtain $u(0,t_j)\ge h_j/4>0$ for a sequence $t_j \to 0$.
\end{proof}

\begin{rem} If $p > 1/2$ and the sequence above is bounded, then the conclusion of Corollary \ref{cor-p9} still holds true.
\end{rem}

Next we investigate the case when $w$ is identically $0$.

\begin{prop}\label{lem-p10} Let  $u$ be  a supersolution in $\om \times [0,T]$ with $p \leq 1/n$.
Assume that $u \geq 0$ at $t=0$ and $u(x,0) >0$ on $\partial \om$.
Then, $u >0$ in $\om$ for any $t >0$.
\end{prop}

\begin{proof} For $p < 1/n$ the proposition follows from Corollary \ref{cor-p9}.

Let  $p = 1/n$. Assume that for $h >0$ small we have $S_h(0)
\subset B_\rho$ for some $\rho $ in $0 < \rho \leq 1$, and
$S_h(0)$ is compactly supported in $\Omega$. We first show that
for $\beta \in (0,\rho]$ small, we have
\begin{equation}\label{eqn-p10}
S_{h/4} (t_0) \subset B_{\rho-\beta}(0), \qquad \mbox{for} \,\,
t_0 = C\, \beta^{1+\frac 1n}.
\end{equation}
To this end, we will apply Proposition \ref{prop-p8} for each $x_0
\in \partial B_\rho$ in the direction $(-x_0)$. Let us assume for
simplicity that $x_0=(0,\cdots,0, -\rho)$. Then, since $S_h(0)
\subset B_\rho$, we have
$$S_h(0) \cap
\{ -\rho < x_n < -\rho+2\beta \} \subset B'_{2\sqrt{\beta}} \times
(-\rho,-\rho+2\beta).$$
Applying Proposition \ref{prop-p8}, we
obtain that
$$S_{h/4}(t_0) \subset \{ x_n > -\rho+\beta\}, \qquad \mbox{for}\, \,\, t_0 = C \, (\beta \, \beta^{\frac{n-1}{2}})^{2/n}.$$
and \eqref{eqn-p10} readily follows.

We will now use \eqref{eqn-p10} to show that $u>0$ for $ t >0$.
Let $t
>0$ and fixed. Choose $\beta := 1/k >0$ with $k$ the smallest
integer so that $C \, \beta^{\frac 1n} \leq  t$, with $C$ the
constant from \eqref{eqn-p10}. Using this $\beta$ we repeat the
argument above $k$ times, starting at $\rho =1$, to conclude that
$$S_{h/4^k}(t_0) \subset B_{1-k\, \beta}, \qquad \mbox{for} \,\, t_0 = C\, k\, \beta^{1+\frac 1n}.$$
This shows that $S_{h/4^k}(t_0) = \emptyset$, for $t_0 =
C\beta^{\frac 1n} \leq t$ hence $u(\cdot,t) \geq h/4^k
>0$.
\end{proof}

\begin{rem}
For $p>1/n$ there exist radial solutions with a flat side that
persists for some time.
\end{rem}

\begin{rem}
In the proof we showed in fact that if $u\ge 0$, $u(x,0) \ge h$ on
$\partial B_1$ then $$u(\cdot , t) \ge he^{-Ct^{-n}}$$ for some
$C$ universal.
\end{rem}

In the next results we investigate the case of angles i.e when
$w(x)=|x_n|$. First proposition shows that $u$ separates instantly
from the edge of the angle when the exponent $p \leq \frac 1 {
n-2}$. The second proposition shows that this is not the case when
$p > \frac 1{n-2}$.

\begin{prop}\label{lem-p11} Assume $u$ is a supersolution,
and  $p \leq \frac 1{n-2}$. If $u(x,0) \geq |x_n|$ and $u(x,0)
>0$ on $\partial \om$, then $u >0$ for any $t >0$.
\end{prop}

\begin{proof} If $p>\frac {1}{n-2}$ then the proposition follows from Corollary \ref{cor-p9} since $a_h \le Ch$.

Let $p=\frac{1}{n-2}$. Since $u\ge |x_n|$ we may assume without
loss of generality that $S_h(0) \subset B_1' \times [-h,h]$. For
each $x_0' \in B_1'$ we apply Proposition \ref{prop-p8} in the
direction $(-x_0')$, in a manner similar to that used in
Proposition \ref{lem-p10}, to show that
$$S_{\frac h4}(t_0) \subset B'_{1-\beta} \times [- h/4, h/4],  \qquad \mbox{for} \,\, t_0=C\, \frac{(\beta\, |S'|)^{2p}}{h^{np-1}}.
$$
Notice that this time $|S'| = \, h\, |B''_{2\, \sqrt{\beta}}|$,
where $B''_r$ is an $n-2$ dimensional ball, hence (since $p=\frac
{1}{n-2}$)
$$t_0 \ge C\, \frac{( h\, \beta^{\frac {n}2})^{2p}}{h^{np-1}}=C \beta^\frac {n}{n-2}.$$
Now the proof continues as in the proof of Proposition
\ref{lem-p10} and we obtain $$u(\cdot, t)\ge
he^{-Ct^{-\frac{n-2}{2}}}.$$
\end{proof}

\begin{prop}\label{prop-p12} There exists a non-trivial solution $u$ of equation
\begin{equation}\label{eqn-pma3}
u_t=(\det D^2 u)^p, \qquad \mbox{on}\,\, \R^n \times [0,\infty)
\end{equation}
for which $u(x,0)\geq |x_n|$ and $u(0,t)=0$, for all $t \in
[0,\delta]$, for some $\delta >0$.
\end{prop}

\begin{proof} We will seek for a solution $u$ of  the form
\begin{equation}\label{eqn-dfnu}
u(x,t) = f(t) \, v(\frac {x}{g(t)})
\end{equation}
for some functions $f=f(t)$ and $v=v(y)$. The function $u$
satisfies \eqref{eqn-pma3} if and only if
$$(-f')\, \left ( \frac xf\,\,  \nabla v(\frac xf) - w \right ) = f^{-n\,p} \, (\det D^2 v)^p.$$
We pick a function $f$ which satisfies
\begin{equation}\label{eqn-eqg}
- f' = f^{-n\, p}.
\end{equation}
Solving \eqref{eqn-eqg}  gives us
\begin{equation}\label{eqn-dfng}
f(t) = [(1+n\, p)\, (T-t)]^{\frac 1{1+np}}
\end{equation}
for any  constant $T >0$. We will next show that there exists a
function $v=v(y)$ such that
\begin{equation}\label{eqn-eqw}
y  \cdot \nabla v - v = (\det D^2 v)^p, \qquad v(y) \geq |y_n|,
\qquad v(0)=0.
\end{equation}
The existence of such a function $v$ implies the claim of our
proposition. To this end, we seek for $v$ of the form
\begin{equation}\label{eqn-dfnw}
v(y',y_n) = \tilde v (|y'|,y_n) = \vphi(|y'|)\,
g(\frac{y_n}{\vphi(|y'|)}),
\end{equation}
with $g(s) \ge |s|$. A direct computation shows that,
$$\tilde v_1 = \vphi' \, g - \vphi' \, \frac{y_n}{\vphi} \, g' = \vphi '\, (g-s\, g'), \qquad \tilde v_2=
g'(\frac{y_n}{\vphi}) =g'(s)$$ with $s={y_n}/{\vphi}$. Also,
$$\tilde v_{11}=\vphi'' \, (g-s\, g') + \vphi' \, s \, g'' \, \frac{y_n}{\vphi^2}\, \vphi', \quad
\tilde v_{12} = - \frac{\vphi'}{\vphi} \, s\, g'', \quad \tilde
v_{22} = \frac 1{\vphi}\, g''.$$ Using that ${y_n}/{\vphi}=s$, we
get
$$y  \cdot \nabla v - v = |y'|\, \vphi'\, (g-s\, g') + y_n \, g' - \vphi\, g =
 (|y'|\, \vphi'-\vphi)\, (g-s\, g'),$$
 and also,
$$\det D^2 v = \frac{\vphi''}{\vphi} \, g'' \, (g-s\, g')^{n-1} \, \left ( \frac{\vphi'}{|y'|}  \right )^{n-2}.$$
Separating the functions $g$ and $\phi$, we conclude that $v$
satisfies \eqref{eqn-eqg}, if
$$g''\, (g-s\, g')^{n-1-\frac 1p} =1 \qquad \mbox{and} \qquad \vphi''   \, \left ( \frac{\vphi'}{|y'|}  \right )^{n-2}
= (|y'|\, \vphi'-\vphi)^{\frac 1p}\, \vphi.$$ For the second
equation we seek for a solution in the form $\vphi(r) = C_{n,p} \,
r^\beta$ with $\beta>1$. We find that $\varphi$ satisfies the
above equation if
$$(\beta -2)\, (n-1) = \frac \beta p + \beta$$
which after we solve for $\beta$ yields to
$$\beta = \frac{2\, (n-1)}{(n-2-1/p)} .$$
Since we need $\beta >1$, we have to restrict ourselves to the
exponents $p > \frac 1{n-2}$.

Next we find an even function $g$, convex of class $C^{1,
\alpha}$, that solves the ODE for $g$ in the viscosity sense and
for which $g(s)=|s|$ for large values of $s$. Rewriting the ODE
and the conditions above in terms of the Legendre transform $g^*$
of $g$ we find
$$(g^*)''=|g^*|^{n-1-1/p} \quad \mbox{in $[-1,1]$}, \quad g^*(1)=g^*(-1)=0.$$
The existence of $g^*$ follows by scaling the negative part of any
even solution $\tilde g$ to the ODE above, i.e $g^*(t)=a\tilde
g(t/b)$ for appropriate constants $a$ and $b$. We obtain the
function $g$ by taking the Legendre transform of $g^*$.

\end{proof}

\begin{rem} Proposition \ref{prop-p12} shows that in the Gauss curvature flow \eqref{eqn-gcf1} with exponent $p$,
if the initial data is a cube, then the edges ($n-1$-dimensional)
move instantaneously if and only if $p \leq \frac{1}{n-2}$.   In
the particular case of the classical Gauss curvature flow with
$p=1$, the edges of the cube move instantaneously if and only $n
\leq 3$.
\end{rem}

\section{The geometry of lines}
Our goal in this section is to prove Theorem \ref{thm-fs1}, which
constitutes the parabolic version of the result of Caffarelli for
Monge-Ampere equation. Theorem \ref{thm-fs1} deals with extremal
points of the set $\{ \, u=0 \, \}$ for a nonnegative solution $u$
of \pma. We begin by giving the definition of an {\em extremal
point} of a convex set (cf. in \cite{Gu}, Chapter 5).

\begin{defn} Let $D$ be a convex subset of $\R^n$. The point $x_0 \in \partial D$ is an
extremal point of $D$ if $x_0$ is not a convex combination of
other points in $\overline D$.

\end{defn}

We now give the main results of this section. The first Theorem
states that a constant segment in time can be extended backward
all the way to the initial data.

\begin{thm}\label{thm-ctt}
Let $u$ be a solution of equation \pma on $\om \times [-T,0]$.
Assume $u(0,t)=0$ for $t\in [-\delta,0]$ and there exists a
section $S_{h_0}(0):=\{u(x,0) < h_0 + p_{h_0}\cdot x\}$ at $(0,0)$
that is compactly supported in $\Omega$. Then $u(0,t)=0$ for all
$t \in [-T,0]$.
\end{thm}

The second Theorem states that if the graph of $u$ at a given time
coincides with a tangent plane in a set $D$ that has an extremal
point in $\Omega$, and $D$ contains at least a line segment, then
$u$ agrees with the initial data on $D$.

 In other words, a line segment at a given time either originates
 from the boundary data at that particular time or from the data at the initial time.

\begin{thm}\label{thm-fs1} Let $u$ be a solution of equation \pma on $\om \times [-T,0]$,
for some convex domain $\om \subset \R^n$. Suppose that at time
$t=0$ we have $u \geq 0$, and the set $$D:=\{ \, u(x,0)=0 \}$$
contains a line segment and $D$ has an extremal point in $\Omega$.
Then,
$$u(x,-T) =0, \qquad  \mbox{for all $x \in D$}.$$
\end{thm}

As a consequence of the theorems above we obtain the following:

\begin{cor}
Assume $u$ is defined in $\om \times [-T,0]$ and $u(x,-T) \ge 0$
on $\partial \om$. Then $u$ is strictly convex in $x$ and strictly
increasing in $t$ at all points $(x,t)$ that satisfy $u(x,-T) <
u(x,t) <0$.
\end{cor}

We first prove Theorem \ref{thm-ctt}.

\begin{proof}[Proof of Theorem \ref{thm-ctt}]
By continuity of $u$ the section
$$S_{h_0}(-\sigma):=\{u(x,-\sigma) < h_0 + p_{h_0}\cdot x\}$$  at
$(0,-\sigma)$ is also compactly included in $\Omega$ for a small
$\sigma \in [0,\delta]$. Let $d$ be sufficiently large so that
$S_{h_0}(-\sigma)$ is $d$-balanced around $0$. By Lemma
\ref{lem-c1a-1}, for each $h \le h_0$ we can find  a section
$S_h(-\sigma)$ which is $C_nd$-balanced around $0$. We apply
Proposition \ref{prop-p6} (see Remark \ref{rem-p6}) and use that
$u(0,0)-u(-\sigma,0)=0<h/2$ to conclude
$$\sigma \le C(d)\frac{|S_h(-\sigma)|^{2p}}{h^{np-1}}.$$

Assume next that $u(0,-t_0)=0$, for some $t_0>\sigma$. We apply
Proposition \ref{prop-p2} (see Remark \ref{rem-p2}) at $(0,-t_0)$
in the set $S:=S_h(-\sigma)$ and conclude
$$u(0,t)\ge -h, \quad \mbox{for} \quad t \ge
-t_0-c(d)\frac{|S_h(-\sigma)|^{2p}}{h^{np-1}}.$$ Using the bound
on $\sigma$ we find that $u(0,t)=0$ for $t \ge -t_0-c(d)\sigma$
and the conclusion follows.
\end{proof}

Next lemma is the key step in the proof of Theorem \ref{thm-fs1}.

\begin{lem}\label{lem-l1}
Assume $u(se_n,0)=0$ for $s\in [0,2]$, and for some $t_0>0$
$$u(e_n,-t_0)\ge -h, \quad \quad T_{6h}(0,-t_0) \subset B_\delta \subset \subset \Omega,$$
where $T_{6h}(0,-t_0)$ is the centered section at $0$ at time
$-t_0$. Then
$$u(e_n, -Mt_0) \ge -2h, \quad \quad \mbox{with $M=1+c \delta^{-2p}$, \quad ($c$ universal)}.$$
\end{lem}

\begin{proof}
Since $u(2e_n, -t_0)\le u(2e_n,0)=0$, the convexity of
$u(\cdot,-t_0)$ implies that $u(0,-t_0) \ge -2h$. We apply
Proposition \ref{prop-p6} (see Remark \ref{rem-p6}) in the section
$$T_{6h}:=T_{6h}(0,-t_0)=\{u(x, -t_0) < u(0,-t_0) + 6h + p_{6h}
\cdot x \, \}$$ and conclude that
$$t_0 \le C \frac{|T_{6h}|^{2p}}{h^{np-1}}.$$
Indeed, otherwise we obtain $u(0,0) \ge h$ which contradicts the
hypothesis.
 Since $T_{6h} \subset B_\delta$ and has $0$ as center of mass, we find
 $$|T_{6h}|\le C \delta |T_{6h}'|, \quad \quad
 \mbox{where $T_{6h}':=\{ x'\in \R^{n-1}| \, \, \, (x',0) \in T_{6h}\}$},$$
for some $C$ depending only on $n$. Using the inequality for $t_0$
we conclude
\begin{equation}\label{eqn-Tbd}
\frac{|T_{6h}'|^{2p}}{h^{np-1}} \ge c \delta^{-2p}t_0.
\end{equation}

 Now we apply Proposition \ref{prop-p2} (see Remark
\ref{rem-p2}) for the function $$\tilde u=u - p_{6h}'\cdot x' -
6h$$ in the convex set $S$ which is the convex hull generated by
the $n-1$ dimensional set $T_{6h}' \times \{0\}$ and the segment
$[0,2e_n] $. Notice that $\tilde u$ is negative at time $-t_0$ in
$S$ and $\tilde u (e_n,-t_0) \ge -7h$. Since $S$ is $d$-balanced
with respect to $e_n$ with $d$ depending only on $n$ we conclude
that
$$\tilde u(e_n,-t) \ge -8h \quad \mbox{for} \quad t \ge -t_0-c
\frac{(2|T_{6h}'|)^{2p}}{h^{np-1}} ,$$ with $c$ universal. Using
\eqref{eqn-Tbd} we find $u(e_n,t) \ge -2h$ if $t \ge
-t_0(1+c\delta^{-2p})$.

\end{proof}

\begin{proof}[Proof of Theorem \ref{thm-fs1}]
Assume for simplicity that $0\in \om$ is an extremal point for $D$
and $2 e_n \in D$. We want to prove that $u(2e_n,-T)=0$.

Fix $\delta >0$ small, smaller than a universal constant to be
specified later. There exists $\sigma>0$ depending on $u$ and
$\delta$ such that $$T_{6h}(0,-t) \subset B_\delta \subset \subset
\Omega \quad \mbox{for all $h,t \in [0, \sigma]$}.$$ Indeed,
otherwise we can find a sequence of $h_n,t_n$ tending to $0$ for
which the inclusion above fails. In the limit we obtain that $0$
can be written as a linear combination of two other points in $D$
(one of them outside $B_\delta$) and contradict that $0$ is an
extremal point.

First we show that $u(x,-\sigma)=0$ on the line segment
$[0,2e_n]$. Using the Holder continuity of $u$ in $t$ at the point
$(e_n,0)$ we find that for small $t_0>0$, $$u(e_n,-t_0) \ge-
h:=-C(u)t_0^{\frac{1}{np+1}}.$$

 We can apply Lemma \ref{lem-l1} inductively and conclude that as long as
 $M^{k-1} t_0 \le \sigma$, $2^{k-1} h \le \sigma$ then
$$u(e_n,-M^kt_0) \ge -2^kh.$$

 We choose $\delta$ small enough so that $M=1+c\delta^{-2p}>4^{np+1}$. Then
 $$2^kh \le C(u)2^{-k}(M^kt_0)^{\frac{1}{np+1}} \le C(u)2^{-k}(M \sigma)^{\frac{1}{np+1}}.$$
 This shows that if we start with $t_0$ small enough then $M^{k-1} t_0 \le
 \sigma$ implies $2^{k-1} h \le \sigma$ and moreover, as $t_0 \to
 0$ then $2^k h \to 0$ as well. We conclude that $u(e_n,
 -\sigma)=0$ hence $u(x,-\sigma)=0$ on the line segment
$[0,2e_n]$.

Now we can use Theorem \ref{thm-ctt} for the points $(se_n,0)$ for
small $s \ge 0$ which are included in a compact section at the
origin at time $t=0$. Since $u(se_n,t)=0$ for $t \in [-\sigma,0]$,
we conclude that $u(se_n, -T)=0$ for small $s$. Then convexity in
$x$ and monotonicity in $t$ imply $u(x,-T)=0$ on the segment
$[0,2e_n]$.

\end{proof}

\section{The geometry of angles}

Our goal in this section is to prove the analogue of Theorem
\ref{thm-fs1} for angles. That is, if $u: \Omega \times [-T,0] \to
\R$ is a solution to \pma \, for which the graph of $u$ at time
$t=0$ has a tangent angle from below, then  this angle originates
either from the initial data $u(\cdot,-T)$ or from the boundary
data on $\partial \om$ at time $t=0$.

Throughout this section we will denote by $x'$ points
$x'=(x_1,\cdots,x_{n-1}) \in \R^{n-1}$ and by $x=(x',x_n)$ points
in $\R^n$. Our result states as follows.

\begin{thm}\label{thm-angle} Let $u:\om \times [-T,0] \to \R$ be a solution of equation \pma
with $\om \subset \R^n$.  Assume  that at time $t=0$, we have
$u(0,0)=0$, $u(x,0) \geq |x_n|$ and  $0$ is an extremal point for
the set $D:=\{ x \in \om: \,\, u(x,0)=0 \}.$ Then, $u(x,-T) \geq
|x_n|$.
\end{thm}

The proof of Theorem \ref{thm-angle} is more involved than that of
Theorem \ref{thm-fs1}. We introduce the following convenient
notation.

\begin{defn}\label{def-A} For negative times $t \le 0$ we say that
$$(h,\alpha) \in A_t(u) \subset \R^2_+$$ if there exist vectors
$q_1,q_2 \in \R^n$ such that
$$u(x,t) \geq u(0,0) -h + \max \, \{ q_1\cdot x,q_2 \cdot x \}$$ in $\Omega$
and $(q_1-q_2) \cdot e_n \geq \alpha$. Whenever there is no
possibility of confusion we write $A_t$ instead of $A_t(u)$.
\end{defn}

\begin{rem}
The statement $(h,\alpha) \in A_t$ is in fact a one-dimensional
condition on $u(x,t)$. It says that, when restricted to the line
$se_n$, we can find a certain angle below the graph of $u(\cdot,t
)$. The vertex of the angle is at distance $h$ below $u(0,0)$ at
the origin and the difference in the slopes of the lines that form
the angle is $\alpha$.

Clearly, if $(h, \alpha) \in A_{t_1}$ then $(h, \alpha) \in A_t$
for all $t \ge t_1$. The statement $(h,\alpha)\in A_t$ remains
true if we add to $u$ a linear function in $x$ or if we perform an
affine transformation in the $x$ variable that leaves $e_n$
invariant.
\end{rem}

Next proposition is the key step in proving Theorem
\ref{thm-angle} and later for obtaining interior $C^{1, \alpha}$
estimates.

\begin{prop}\label{prop-angle-main} Let  $u$ be   a solution of equation \pma
with $u(0,0)=0$.  Assume that at time $-t_0$, ($t_0>0$) the
solution $u$ satisfies for a fixed constant $C_0$ and a parameter
$\delta \le 1$:
\begin{enumerate}[i.]
\item $(h,\alpha) \in A_{-t_0}$ and $(C_0\, h, (1+\delta) \, \al ) \notin A_{-t_0}$,
and
\item there exists a section (at distance $h$ from the origin)
$$S_h := \{   u(x,-t_0) < h + q\cdot x \, \}$$ of $u(\cdot,-t_0)$ which is $d$-balanced with respect to the
origin and is compactly supported in $\om$.
\end{enumerate}
Then,
$$\left (C_0^2\, h, \frac {\al}{1+ \frac \delta 2}  \right ) \in A_{-t}, \qquad \mbox{for} \quad  t_0 \leq t \leq t_0 +c(d)  \, \delta^{-2p}\, t_0$$
for some $c(d) >0$.

\end{prop}

\begin{rem}
From the proof we will see that we can take the constant
$C_0=100$.
\end{rem}

\begin{proof} Since   $(h,\alpha) \in A_{-t_0}$,  we have
$u(x,t) \geq - h + \max \, \{ q_1\cdot x,q_2 \cdot x \}$, for some
vectors $q_1, q_2$. Without loss of generality, we may assume that
$q_1, q_2$ have only components in the $e_n$ direction. This
reduction is possible by first subtracting the linear map
$\frac{q_1+q_2}{2}\cdot x$ and then performing a linear
transformation that leaves $e_n$ invariant. Thus, assume that
$$u(x,-t_0) \geq -h + \frac \alpha 2 \, |x_n|.$$
Since $S_h$ is $d$-balanced, the inequality above and Remark
\ref{rem-fs1} imply that
$$S_{h} \subset \{  \, |x_n| < 4d \, \frac {h}{\alpha} \, \}.$$
Thus, if $S'_h:= S_h \cap \{ \, x_n =0\, \}$, we have
$$
|S_h| \leq Cd \, \frac {h}{\alpha} \, |S_h'|.$$ Since $u(0,0) =0$
and $u(0,-t_0) \geq -h$, Proposition \ref{prop-p6} implies that
$$t_0 \leq C(d) \, \frac{|S_h|^{2p}}{h^{np-1}},$$
and from the previous estimate we have
\begin{equation}\label{eqn-angle2}
t_0 \leq C(d) \, \frac{( |S'_h| \, \frac h\alpha)^{2p}}{h^{np-1}}.
\end{equation}

On the other hand, since $(C_0\, h, (1+\delta) \, \al ) \notin
A_{-t_0}$ there exists $s_1 e_n \in \Omega$ with $s_1 >0$, such
that
$$u(s_1\, e_n, - t_0) < - C_0\, h + \frac \al 2 \, (1+2\, \delta)\, s_1.$$
Otherwise the angle with vertex at $-C_0$ and lines of slopes
$-\al/2$, $\al/2 + \delta \al$ would be below the graph of $u(x,
-t_0)$ on the line $x=se_n$ and we reach a contradiction.

Since $u(s_1\, e_n, - t_0) \geq - h + \frac \al 2\, s_1$, the
above yields the bound
$$s_1 \geq \frac{(C_0-1)\, h}{\al \, \delta}:=s_0.$$
Moreover, since $u(0,-t_0) \leq u(0,0) \leq 0$ and $$u(s_1\, e_n,
- t_0) < - C_0\, h + \frac \al 2 \, (1+2\, \delta)\, s_1 < \frac
\al 2 \, (1+2\, \delta)\, s_1$$ the convexity of $u(\cdot,-t_0)$
implies that
$$u(s \, e_n, - t_0) < \frac \al 2 \, (1+2\, \delta)\, s, \qquad \forall s \in [0,s_0]\subset [0,s_1].$$
Hence, if $s \in [0,s_0]$, then
$$u(s \, e_n, - t_0) < \frac \al 2 \, s + \al \, \delta \, s_0 = (C_0-1)\, h + \frac \alpha 2 \, s.$$
Recalling that $ S_h := \{ u(x,-t_0) < h + q'\cdot x' + q_n \, x_n
\, \}$, it follows from the above discussion that the set
$$ \{  u(x,-t_0) < (C_0-1)\, h + q'\cdot x' + \frac \al 2 \, x_n \, \}$$
contains the convex set $\tilde S$ which is generated by $S_{h}':=
S_h \cap \{\, x_n =0 \}$ and the segment $ [0,s_0e_n]$. It follows
from the convexity of $\tilde S$ that
\begin{equation}\label{eqn-angle3}
|\tilde S| \geq c_n  \, |S_h'| \, s_0 =c_n \, |S_h'| \,  \frac {
(C_0-1)\, h} {\al \, \delta}
\end{equation}
for some universal $c_n>0$.

We apply  Proposition \ref{prop-p2} (see Remark \ref{rem-p2}) for
$\tilde S$ which is $Cd$-balanced around $s_0e_n /2$ and with
$\tilde h=C_0h$, $\tilde \delta=1/30$, and find that (since $C_0
\ge 100$)
$$u(\frac{s_0\, e_n}2,-t) \geq - h + \frac \al 2 \, \frac{s_0}{2} - \frac{C_0\, h}{30} \geq \frac \al 2\, (1-\frac \delta 5) \, \frac {s_0}2$$
for
$$ - t_0 - c(d)\,
\frac {|\tilde S|^{2p}}{h^{np-1}} \leq -t \leq -t_0.
$$
Observing that a similar consideration holds for  negative $x_n$
and using \eqref{eqn-angle3} we conclude
$$u(\pm \frac{s_0\, e_n}2,-t)  \geq \frac \al 2\, (1-\frac \delta 5) \, \frac {s_0}2$$
for
$$ - t_0 - c(d)\,
 \frac{\left ( |S_h'| \, \frac { (C_0-1)\, h} {\al \, \delta} \right )^{2p} }{h^{np-1}} \leq -t \leq
 -t_0,
$$
or, from \eqref{eqn-angle2}, for $$ -t_0-c(d)\delta^{-2p}t_0 \le
-t \le -t_0.$$ It follows that for such $t$ we have (since
$u(0,-t)\le 0$)
$$\nabla u(\pm \frac{s_0\, e_n}2,-t) \cdot ( \pm e_n)  \geq \frac  \al 2\, (1-\frac \delta 5). $$
Setting
$$\tilde q_1 = \nabla u( \frac{s_0\, e_n}2,-t) \quad \mbox{and} \qquad \tilde q_2 = \nabla u(-\frac{s_0\, e_n}2,-t)$$
we obtain
$$ (\tilde q_1 - \tilde q_2) \cdot e_n \geq \al \,  (1-\frac \delta 5) \geq \frac \al{1 + \frac \delta 2}$$
since $\delta \leq  1$. From the convexity of $u(\cdot, -t)$ and
the inequalities
$$u(s_0\, e_n, -t) \leq u(s_0\, e_n, -t_0) \leq \frac \alpha 2\,
s_0 + (C_0-1)\, h$$
$$u(\frac{s_0e_n}{2},-t) \ge \frac \alpha 2\, \frac{s_0}{2}-\frac{C_0-1}{20}h  $$
we conclude that the tangent planes at $\pm \frac{s_0\, e_n}2$ for
$u(\cdot, -t)$ are above $-2C_0h$ (and therefore $-C_0^2 h$) at
the origin. This implies that
$$\left (C_0^2\, h, \frac {\al}{1+ \frac \delta 2}  \right ) \in A_{-t}, \qquad \mbox{if} \quad t_0 \leq t \leq t_0 +c(d)  \, \delta^{-2p}\, t_0$$
which finishes the proof of the proposition.

\end{proof}

\begin{rem}\label{rem-p_angle}
If hypothesis ii) is satisfied only for a time $-\tilde t$ with
$\tilde t \le t_0$ i.e
$$S_h:=\{u(x,-\tilde t) \le h + q \cdot x \} \subset \subset \om \quad \quad \mbox {and $S_h$ is $d$-balanced around 0,}$$
then the same conclusion holds in the smaller time interval
$$\left (\, h , \frac \al{1+ \frac \delta 2} \right ) \in A_{-t},
\qquad \mbox{for}\,\,\, t_0 \leq t \leq t_0+ c(d) \,
\delta^{-2p}\, \tilde t. $$

Indeed, the only difference appears when estimating $|S_h'|$ from
below: in (\ref{eqn-angle2}) we have to replace the left hand side
$t_0$ by $\tilde t$.

\end{rem}

\begin{rem}\label{rem-angle5} If $x^*$ denotes the center of mass of the $d$-balanced section $S_h$ at time $-t_0$, then  it follows from the proof of Proposition
\ref{prop-angle-main} and Remark \ref{rem-fs1} that
$$u(x,-t) \geq u(x^*,-t_0) - C(d) \, h + \max_{i=1,2} \, \{\tilde q_i \cdot (x-x^*) \}$$
for $t_0 \leq t \leq t_0 + c(d)  \, \delta^{-2p}\, t_0$, with
$$(\tilde q_1 - \tilde q_2) \cdot e_n \geq \frac {\alpha}{1+ \frac \delta 2}, \qquad \quad \tilde q_i = \nabla u \, (\frac {s_0e_n}2, -t ).$$
In other words, if $\tilde u$ is the translation of $u$ defined by
\begin{equation*}
\tilde u(x,t) = u(x+x^*, t-t_0) - u(x^*,-t_0)
\end{equation*}
then
$$\left (C(d)\, h , \frac \al{1+ \frac \delta 2} \right ) \in A_{-t}(\tilde u), \qquad \mbox{for}\,\,\,
0 \leq t \leq c(d)  \, \delta^{-2p}\, t_0. $$
\end{rem}

\smallskip

We will now proceed to the proof of Theorem \ref{thm-angle}.

\begin{proof}[Proof of Theorem \ref{thm-angle}] We will denote throughout the proof
by $u_0:=u(\cdot,0)$. Since $u_0 \geq 0$, and $0$ is an extremal
point for the set $D=\{ u_0 =0 \, \}$ we can find (as in the proof
of Theorem \ref{thm-fs1}) $\sigma_0:=\sigma_0(u)
>0$ small, depending on $u$, such that if $0 \leq h,t \leq
\sigma_0$ then the section $$T_{h,-t}:=\{u(x,-t)\le h + q \cdot
x\}$$ of $u(\cdot,-t)$ that has $x=0$ as center of mass is
compactly supported in $\om$. Thus, by John's lemma  $T_{h,-t}$ is
$C_n$-balanced with respect to the origin.

Let $ 0 < \delta < \delta_0$ with $\delta_0$ small universal
constant to be made precise later. Without loss of generality we
may assume that $u_0$ is tangent to $|x_n|$ on the line $x'=0$ at
the origin, i.e. we have
\begin{equation}\label{eqn-lip}
\lim_{x_n \to 0^+} \frac{u_0(0,x_n)}{x_n}=1 \qquad \mbox{and}
\qquad \lim_{x_n \to 0^-} \frac{u_0(0,x_n)}{x_n}=-1.
\end{equation}
Hence, by  taking $\sigma_1=\sigma_1(\delta,u)$ smaller than
$\sigma_0$, depending also on $\delta$, we can assume that
$$\left (\tilde h, 2\, (1+\frac \delta 2) \right ) \notin A_0, \qquad \mbox{for} \,\,\, \tilde h \leq \sigma_1.$$
Choose $h << \sigma_1$. Since $u_0$ is Lipschitz  in say $B_a
\subset \Omega$ with $|\nabla u_0| < 1/a$, for some small $a$ we
find (using Proposition \eqref{prop-p2}) that at time $-t_0$,
given by
$$t_0:=c(a) \, \frac{h^{n2p}}{h^{np-1}}= c(a)\, h^{np+1}$$
the we have $u(x,-t_0) \ge u_0(x) -h$ for $x \in B_a$. This easily
implies
\begin{equation}\label{eqn-angle-step0}
(h,\al) \in A_{-t_0}, \qquad \al:=2\, (1-\frac 1a \,h).
\end{equation}
Also notice that
$$\left (\tilde h, \al \, (1+ \delta) \right ) \notin A_0, \qquad \mbox{if} \,\,\, h, \tilde h \leq \sigma_2=\sigma_2(a,\sigma_1).$$
We choose $\delta_0$ such that
$$M^2:= c(C_n) \, \delta^{-2p} \geq c(C_n) \, \delta_0^{-2p}:=C_0^{10(np+1)}$$
where $c(d)$ is the constant that appears in Proposition
\ref{prop-angle-main}.

\begin{lem}
As long as $M^k \, t_0 \leq \sigma_0$ and $C_0^{3k+1}\, h \leq
\sigma_2$, there exists $0 \leq m \leq k$ such that
\begin{equation}\label{eqn-angle-claim1}
\left ( C_0^{3k-m}\, h, \alpha \,\,  \frac 1{1+\frac \delta 2}
\cdots  \frac 1{1+\frac \delta {2^m}} \right  )\in A_{-M^k\, t_0}.
\end{equation}
\end{lem}
\begin{proof}We will use induction in $k$. When $k=0$ we take $m=0$ and
we use \eqref{eqn-angle-step0}. Assume now that the statement
holds for $k$ and let $m$ be the smallest so that
\eqref{eqn-angle-claim1} holds. If $m >0$, then
$$\left ( C_0^{3k-(m-1)}\, h, \alpha \, \, \frac 1{1+\frac \delta 2}  \cdots  \frac 1{1+\frac \delta {2^{m-1}}} \right  )\notin A_{-M^k\, t_0}.$$
Combining this with \eqref{eqn-angle-claim1}, and applying
Proposition \ref{prop-angle-main} we find that
$$\left ( C_0^{3k-m+2}\, h, \alpha \, \frac 1{1+\frac \delta 2}  \cdots
\frac 1{1+\frac \delta {2^{m+1}}} \right  )\in A_{-t}, \qquad
\mbox{if}\,\,\,  t \leq M^{k+2}\, t_0$$ which proves
\eqref{eqn-angle-claim1}  for the pair $(k+1,m+1)$.

If $m=0$, then $( C_0^{3k}\, h, \alpha ) \in A_{-M^k\, t_0}$. On
the other hand, since $C_0^{3k+1}\, h \le \sigma_2$ we have $ (
C_0^{3k+1}\, h, \alpha\, (1+\delta) ) \notin A_0$, thus $$(
C_0^{3k+1}\, h, \alpha\, (1+\delta) ) \notin A_{-M^k\, t_0}.$$
Hence, by Proposition \ref{prop-angle-main}
$$ (C_0^{3k+2}\, h, \alpha\, \frac 1{1+\frac \delta 2}   ) \in A_{-t}$$
for $t \leq M^{k+2}\, t_0$ which again proves
\eqref{eqn-angle-claim1}  for the pair $(k+1,1)$. This concludes
the proof of the lemma.
\end{proof}
We will now finish the proof of the theorem. Since $M \geq
C_0^{5(np+1)}$ and $t_0=c\, h^{np+1}$ we see that  for the last
$k$ for which $M^k t_0 \leq \sigma_0$ we satisfy
$$C_0^{3k+1} h \leq C_0 M^{\frac 3 5 \frac{k}{np+1}}h \le C(\sigma_0) \, h^\frac{2}{5} < \sigma_2$$
if  $h << \sigma_2$ is sufficiently small. Also, if $\delta$ is
chosen small, depending on $\sigma_0$ and $T$, for the last $k$ we
also have  $M^{k+2} \, t_0 \geq T$. We conclude from the lemma
above that
$$(C(\sigma_0)\, h^\frac{2}{5}, \alpha\, e^{-\delta}) \in A_{-T}$$
and by letting $h \to 0$ we obtain
$$(0, 2\, e^{-\delta}) \in A_{-T}.$$
Finally, letting $\delta \to 0$ we conclude that $(0,2) \in
A_{-T}$ which proves the theorem.
\end{proof}

\section{$C^{1,\alpha}$ regularity - I}\label{sec-c1ad}

In the next two sections we establish $C^{1,\alpha}$ interior
regularity of solutions to \eqref{eqn-pma}. They are based on
quantifying the result of Theorem \ref{thm-angle}. In the elliptic
case $C^{1,\alpha}$ regularity is obtained by a compactness
argument. However, in our setting compactness methods would only
give $C^1$ continuity for exponents $p\le \frac{1}{n-2}$. The
reason for this is that in the parabolic setting it is more
delicate to normalize a solution in space and time.

The main result of this section is the following Theorem (see
Definition \ref{def-c1a}).

\begin{thm}\label{thm-c1ad}
Let $u$ be a solution to \eqref{eqn-pma} in $\om \times [-T,0]$
and assume there exists a section of $u(x,0)$ which is
$d$-balanced around $0$ and is compactly supported in $\om$.

a)If the initial data $u(x,-T)$ is $C^{1,\beta}$ at $0$ in the $e$
direction then $u(x,0)$ is $C^{1, \alpha}$ at the origin in the
$e$ direction with $\alpha=\alpha(\beta,d)$ depending on $\beta$,
$d$ and the universal constants.

b) If $u(0,0)>u(0,-T)$ then $u(x,0)$ is $C^{1, \alpha}$ at the
origin with $\alpha=\alpha(d)$ depending on $d$ and the universal
constants.
\end{thm}

Part $b)$ will be improved in Theorem \ref{thm-c1a4} in which we
show that $\al$ can be taken to be a universal constant. As a
consequence we obtain Theorem \ref{thm-fb}.

\begin{proof}[Proof of Theorem \ref{thm-fb}] In view of Remark \ref{rem-fs1}, at a point $(x,t)$
for which $u(x,t) \le c_n$, with $c_n$ small depending only on $n$
the centered section $T_{h}(x,t)$ at $x$, for small $h$, is
compactly supported in $\Omega$. Clearly $u(x,0)$ is $C^{1,1}$ at
an interior point of the set $\{u(x,0)=0\}$. Thus we can apply
Theorem \ref{thm-c1ad} with $d$ depending only on $n$ and
$\beta=1$ and obtain the desired result. If $c_n<u(x,t)<1$ then we
can apply directly Theorem \ref{thm-c1a4} and obtain the same
conclusion. The second part of the theorem follows similarly.

\end{proof}

The following simple lemma gives the relation between the sets
$A_t$ defined in Definition \ref{def-A} and $C^{1,\al}$
regularity. Its proof is straightforward and is left to the
reader.

\begin{lem}\label{c1a-lem7} Let  $f: \R \to \R$ be a convex function with $f(0)=0$ and let $q$ be
a sub-gradient of $f$ at $x=0$. If,  for some $x$,  we have $f(x)
- q \cdot x \geq a\, |x|^{1+\al}$, then
\begin{equation}\label{eqn-c1a-9}
(h,a^{\frac 1{1+\al} } \, h^{\frac \al{1+\al}} ) \in A (f)
\end{equation}  with $ h=a\, |x|^{1+\al}.$
Conversely, if for some number $h$,  \eqref{eqn-c1a-9} holds, then
$$f(x) - q \cdot x \geq \frac a{4^{\al +1}} \, |x|^{1+\al}$$ for
some $x$ with $|x| = 4\, ( \frac ha )^{\frac 1{\al +1}}$.
\end{lem}

 As a consequence we obtain the following useful corollary.
\begin{cor}\label{cor-c1a}
The function $u(x,0)$ is $C^{1, \alpha}$ at $0$ in the $e_n$
direction if and only if $$(h, C h^\frac{\alpha}{\alpha+1}) \notin
A_0$$ for some large $C$ and for all small $h$.
\end{cor}

Theorem \ref{thm-c1ad} will follow from the following lemma.

\begin{lem}\label{c1a-l3}  Assume that $u: \om \times [-T,0] \to \R$  is a solution of \eqref{eqn-pma} such that $u(0,0)=0$, $u(x,-T)> 1$ on $\partial \om$, and
\begin{equation}\label{eqn-c1a-6}
B_{\frac 1d} \subset \{ \,  u(x,0) < 1 \, \} \subset \{ \, u(x,-T)
< 1 \, \} \subset B_1.
\end{equation}
Choose $\delta_0(d)$ sufficiently small, so that
\begin{equation}\label{eqn-c1c-11}
c(C_nd) \, \delta_0^{-2p} = C_0^{12\, (np+1)}:=M,
\end{equation}
where $c(C_nd)$ and $C_0$ are the constants from Proposition
\ref{prop-angle-main} and $C_n$ the constant from Lemma
\ref{lem-c1a-1}. Assume also that $(C_0^{-k}, (1+\delta_0)^{-l})
\in A_{-t_0}$, for some $k \geq 0$ and some $0<t_0 \le T$.

 There exists a
constant $C_1(d)
>0$ such that if $m_0$ is an integer satisfying $$3 \, m_0 \le k -l- C_1(d) \quad \quad \mbox{and} \quad  M^{m_0} \, t_0 \geq T,$$ then
$$\left(C_0^{C_1(d) + l + 3\, m_0-k}, (1+\delta_0)^{-l-C_1(d)}\right) \in A_{-T}.$$
\end{lem}

\begin{proof} Define $\eta: \N \to \Z$ as
$$(C_0^{-k}, (1+\delta_0)^{-\eta(k)-1}) \in A_{-t_0} \quad \mbox{but} \quad (C_0^{-k}, (1+\delta_0)^{-\eta(k)}) \notin A_{-t_0}.$$
Clearly,
\begin{enumerate}[i)]
\item $\eta$ is nondecreasing i.e $\eta (k+1) \geq \eta(k)$,
\item $\eta(0) \geq -C_1(d)$, and
\item $\eta(k) < l$ (by assumption).
\end{enumerate}
For each integer  $m$ with $0 \leq m \leq \frac{k-l-C_1(d)}3$, we
define $s_m$ as the largest $s$, $0 \leq s \leq k$ that satisfies
$$\eta(k-s) \leq l+3m-s.$$
Notice that we satisfy the inequality above when $s=0$ and the
opposite inequality when $s=k$. We obtain:
\begin{equation}\label{eqn-s_m}
\eta(k-s_m) = l+3m-s_m, \quad \mbox{thus} \quad s_m -3m \leq
l+C_1(d).
\end{equation}
 Also, from the definition of $s_m$ we find that $s_{m+1} \geq s_m +3$.

{\bf Claim: }{\it There exists $(r_1,r_2,r_3) \in \Z^3$, $r_i \geq
0$, such that
\begin{equation}\label{eqn-r123}
\left ( C_0^{r_1-k}, (1+\delta_0)^{r_2-l} \, \frac 1{ (1+ \frac
{\delta_0}2)  \cdots (1+ \frac {\delta_0}{2^{r_3}}) } \right ) \in
A_{-t_m}, \quad \quad \quad t_m=M^mt_0
\end{equation}
 with
\begin{equation}\label{eqn-c1a-4}
r_1-r_2+r_3 = 3m, \quad r_3 \leq m, \quad r_1+r_3 \leq s_m \, ( \,
\Leftrightarrow 0 \leq r_2 \leq s_m - 3m).
\end{equation}}
{ \it Proof of Claim:} In order to simplify the notation, instead
of (\ref{eqn-r123}) we write
$$(r_1,r_2,r_3) \in  \mathcal {A}_{-t_m} $$

We will use induction on $m$. For $m=0$ the claim holds from our
assumption  $(C_0^{-k}, (1+\delta_0)^{-l}) \in A_{-t_0}$, if
$(r_1,r_2,r_3)=(0,0,0)$.

Assume now that the claim holds for $m$. Consider the pairs
$$ ( r_1+s, r_2 , r_3-s), \qquad \mbox{if} \,\,\, 0 \leq s \leq r_3$$
$$(r_1+s, r_2+s-r_3,0), \qquad \mbox{if} \,\,\,  s \geq r_3$$
where $(r_1,r_2,r_3)$ comes from the induction step $m$. When
$s=0$ the first pair belongs to $\mathcal{A}_{-t_m}$,  by the
induction hypothesis, and when $s=s_m - r_1$ the second pair
doesn't belong to $\mathcal{A}_{-t_m}$,  since for that choice of
$s$ the second pair is
$$\left (C_0^{s_m-k}, (1+\delta_0)^{-(l+3m -s_m)} \right ) =\left (
C_0^{-(k-s_m)}, (1+\delta_0)^{-\eta(k-s_m)} \right )  \notin
A_{-t_0}$$ from the definition of the function $\eta$ given above.
Note that for $s=r_3$ the two pairs are the same.

It follows that either there exists an $s < r_3$ such that
$$( r_1+s, r_2, r_3-s)  \in \mathcal{A}_{-t_m} \quad \mbox{
and} \quad  (r_1+s+1,r_2, r_3-s-1) \notin \mathcal{A}_{-t_m}$$ or,
there exists an $ r_3 \leq s < s_m-r_1$ such that
$$(r_1+s,r_2+s-r_3,0) \in \mathcal {A}_{-t_m} \quad
\mbox{and} \quad (r_1+s+1,r_2+s+1-r_3,0 ) \notin \mathcal A_{-
t_m}.$$

In either case we can apply Proposition \ref{prop-angle-main}.
Indeed, the hypothesis (\ref{eqn-c1a-6}) and Lemma \ref{lem-c1a-1}
imply the existence of a section $S_h$ of $u(\cdot, t)$ that
satisfies ii) in Proposition \ref{prop-angle-main} for any $h \le
1$ and any $t \in [-T,0]$. More precisely, $S_h$ is
$C_nd$-balanced section around $0$ and it is compactly supported
in $\om$. We conclude that either $(r_1+s+2,r_2,r_3-s+1)$ for some
$0 \leq s < r_3$ or $(r_1+s+2,r_2+s-r_3, 1)$ for some $s \geq r_3$
belongs to $\mathcal A_{-M t_m}$. Notice that in both cases the
sum of the first and third component is less than $s_m +3 \le
s_{m+1}.$ This concludes the proof of the claim.

\

The lemma follows now from the claim above. Since $M^{m_0}\, t_0
\geq T$ and $$r_1\le s_{m_0} \le l + 3 m_0 + C_1(d), \quad \quad
r_2 \ge 0,$$ we conclude that
$$\left ( C_0^{C_1(d)+l+3m_0 -k}, (1+\delta_0)^{-l} \,e^{-\delta_0}   \right ) \in A_{-T}.$$

\end{proof}

\begin{rem}\label{rem-c1ad}
If we assume that hypothesis \eqref{eqn-c1a-6} holds only on a
smaller interval $t\in[-T_1,0]$ instead of the full interval
$[-T,0]$ then the same conclusion holds by replacing $C_1(d)$ with
a constant $C_1(d, T/T_1)$.

 The only difference occurs in the inductive step that
shows
 $(r_1,r_2,r_3) \in \mathcal{A}_{-t_m},$
and we have to distinguish whether $t_m \le T_1$ or $t_m > T_1$.
The case when $t_m \le T_1$ is the same and we obtain
$t_{m+1}=Mt_m$ as before. In the case when $t_m>T_1$ we apply
Remark \ref{rem-p_angle} of Proposition \ref{prop-angle-main} and
obtain $t_{m+1}=t_m +MT_1$. This second case occurs at most
$T/(MT_1)=C(d,T/T_1)$ times and therefore we need to replace $m_0$
by $m_0+C(d, T/T_1)$.

\end{rem}

\begin{rem}\label{c1a-r4}
If in the assumption \eqref{eqn-c1a-6} we have a constant $a$
instead of $1$ i.e
$$B_{\frac 1d} \subset \{ \, x: \,\, u(x,0) < a \, \} \subset \{ \, x: \,\, u(x,-T) < a \, \} \subset B_1
$$
then the conclusion is the same, except that $k \geq 0$ is
replaced by $k \geq C(a)$ and $C_1(d)$ is replaced by $C_1(d,a)$.

Indeed, $\tilde u (x,t) := \frac 1a \, u(x,a^{1-np}\, t)$
satisfies the assumptions  of the lemma with $\tilde t_0 =
a^{np-1} \, t_0$ and $\tilde T = a^{np-1}\, T$ and
$(C_0^{-k+C(a)}, (1+\delta_0)^{-l-C(a)} ) \in A_{-\tilde
t_0}(\tilde u)$, hence  the conclusion of the lemma follows.
\end{rem}

Next we prove Theorem \ref{thm-c1ad}.

\begin{proof}[Proof of Theorem \ref{thm-c1ad}] From the
continuity of $u$ we can assume that, after a linear
transformation, we have the following situation: $u(0,0)=0$,
$u(x,-T_1)>1$ on $\partial \om$ and
$$ B_{\frac {1}{2d}} \subset \{ \,  u(x,0) < 1 \, \} \subset
\{ \, u(x,-T_1) < 1 \, \} \subset B_1  $$ for some small $T_1\in
(0,T]$.

Let $k \ge 0$, $l$ be integers such that $$(C_0^{-k},
(1+\delta_0)^{-l})\in A_0.$$

In view of Corollary \ref{cor-c1a} it suffices to show that there
exists $\eps:=\eps(d, \beta)$ small (or $\eps=\eps(d)$ for the
second part) such that $l \ge \eps k$ for all large $k$. Assume by
contradiction that
$$l < \eps k \quad \mbox{for a sequence of $k \to \infty$.}$$
Then, from the Lipschitz continuity of $u(x,0)$ in $B_{1/4d}$ and
Proposition \ref{prop-p2} we find (as in the proof of Theorem
\ref{thm-angle}) that
$\left(2C_0^{-k},(1+\delta_0)^{-l}-C(d)C_0^{-k}\right) \in
A_{-t_0}$ or, for $k$ large enough
\begin{equation}\label{eqn-lip1}
(C_0^{1-k}, (1+\delta_0)^{-l-1}) \in A_{-t_0} \quad  \mbox{with}
\quad t_0:=c(d)C_0^{-k(np+1)}.
\end{equation}

 Now we can apply Remark \ref{rem-c1ad} and conclude that if
 $$3m_0 \le k-l-C_1 \quad \mbox{and} \quad M^{m_0}t_0\ge T$$ then
$$(C_0^{C_1+3m_0+l-k}, (1+\delta_0)^{-l-C_1} ) \in A_{-T} \quad \mbox{with $C_1=C_1(d,T/T_1)$}. $$
We choose $m_0=[\frac k6]$ to be the smallest integer greater than
$k/6$. Clearly both inequalities for $m_0$ are satisfied for $k$
large (we assume $\eps \le 1/6$) since $M=C_0^{12(np+1)}$ and
$$M^{m_0}t_0\ge C_0^{2k(np+1)} t_0 \to \infty \quad \mbox{as $k \to \infty$.}$$
Thus $$(C_0^{-k/6}, (1+\delta_0)^{-2\eps k}) \in A_{-T} \quad
\mbox{for a sequence of $k \to \infty$.}$$

We reached a contradiction if $u(0,-T)<0$ (we choose $\eps=1/6$).

If we assume that $u(0,-T)=0$ and $u(x,-T)$ is $C^{1, \beta}$ at
$0$ in the $e_n$ direction then it follows from Corollary
\ref{cor-c1a},
$$ \frac{\log{C_0}}{6} \frac{\beta}{\beta +1} \le  2 \eps \log(1+\delta_0)$$
and we reach a contradiction again by choosing $\eps(d,\beta)$
small.
\end{proof}

\section{$C^{1,\alpha}$ regularity - II}\label{sec-c1a}

In this section we prove the main estimates. Let $u$ be a solution
defined in $\om \times [-T,0]$ and assume that $u>l(x)$ on
$\partial \om \times [-T,0]$ for some linear function $l(x)$. We
are interested in obtaining $C^{1,\alpha}$ estimates in $x$ at
time $t=0$ in any compact set $K$ included in the section
$\{u(x,0)< l(x)\}$. Theorem \ref{thm-c1ad} gives such estimates
but with the exponent $\al$ depending also on the distance from
$K$ to $\partial \{u(x,0)< l(x)\}$ which is not desirable.

We can assume that after rescaling we are in the following
situation:
\begin{equation}\label{eqn-pma81}
\la \, \ddua \leq u_t \leq \La \, \ddua, \qquad \mbox{in} \,\,
\Omega \times [-T,0],
\end{equation}
\begin{equation}\label{eqn-pma82}
\mbox{$u > 1$ on $\partial \om \times [-T,0]$, \quad \quad $\om
\subset B_1(y)$ for some $y\in \R^n$, }
\end{equation}
\begin{equation}\label{eqn-pma83}
\mbox{ $u_0(x):=u(x,0)$ satisfies $u_0(0)=0$.}
\end{equation}

 First two theorems
deal with the case $p<\frac{1}{n-2}$ and $p=\frac{1}{n-2}$. In
view of the results of Section 3, $C^{1, \alpha}$ (or $C^1$)
continuity is expected for these exponents regardless of the
behavior of the initial data at time $-T$.

\begin{thm}\label{thm-c1a2} Let $u$ be a solution of \eqref{eqn-pma81}-\eqref{eqn-pma83} with $0 < p <
\frac 1{n-2}$ and $T \le 1$. Then,
$$\|u_0\|_{C^{1,\alpha}(K)} \leq C(K) \, T^{-\gamma}
\quad \mbox{for any set}\quad K \subset \subset \{\,  u_0(x)  < 1
\}.$$ The constants $\alpha, \gamma >0$ are universal (depend only
on $n, p, \la$ and $\La$), and $C(K)$ depends on the universal
constants and the distance between $K$ and $\partial \{ u_0(x) < 1
\, \}$.
\end{thm}

 The example in Proposition \ref{prop-p12} shows that the Theorem
\ref{thm-c1a2} fails when $ p > \frac 1{n-2}$. For the critical
exponent $p=\frac 1{n-2}$ we obtain a logarithmic modulus of
continuity of the gradient.

\begin{thm}\label{thm-c1a3} Under the same assumptions and notation as in Theorem \ref{thm-c1a2}, if $p=\frac 1{n-2}$, then
$$|\nabla u_0(x) - \nabla u_0( y) | \leq C(K)\,  \, |\log |x-y||^{-\al} \, T^{-\gamma}, \quad \quad \forall x,y\in K.$$
\end{thm}

Next two theorems deal with the case of general exponents $p>0$.
First theorem states that if the initial data $u(x,-T)$ is $C^{1,
\beta}$ in the $e$ direction then $u(x,0)$ is $C^{1,\alpha}$ in
the $e$ direction with $\al=\al(\beta)$.

\begin{thm}\label{thm-c1a1} Let $u$ be a solution of \eqref{eqn-pma81}-\eqref{eqn-pma83} with $p>0$.
If  $$ \partial _e u(\cdot,-T) \in C^ \beta(\bar{S}), \quad \quad
S:= \{\, u(x,-T) < 1 \},$$ for some $\beta>0$ small, then for any
set $K \subset \subset \{\,  u_0(x)  < 1 \}$
$$\|\partial _e u_0\|_{C^{\alpha}(K)} \leq C(K)\|\partial_e u(\cdot,-T)\|_{C^\beta(\bar S)}. $$ The constant $\alpha=\al(\beta) >0$ depends on $\beta$ and the universal
constants.
\end{thm}

The second Theorem is a pointwise $C^{1,\alpha}$ estimate at
points that separated from the initial data at time $-T$.

\begin{thm}\label{thm-c1a4}
Let $u$ be a solution of \eqref{eqn-pma81}-\eqref{eqn-pma83} with
$p>0$. If
$$u(0,0) - u(0,-T) := a >0$$ then, there exists $q \in \R^n$ for
which
$$|u_0(x)  - q \cdot x | \leq C(a) \, |x|^{1+\al}$$
with $\al$ universal  and $C(a)$ depends on $a$, the distance from
$0$ to $\partial \, \{ u_0 < 1 \} )$ and the universal constants.

\end{thm}

The theorems above will follow from a refinement of Lemma
\ref{c1a-l3}. We show that we may choose $\delta_0$ universal in
Lemma \ref{c1a-l3} and satisfy the conclusion at a point $\tilde
x$ possibly different from the origin. The key step is to use the
part $b)$ of Lemma \ref{lem-c1a-1}.

\begin{lem}\label{c1a-p5} Let $u: \om \times [-T,0] \to \R$  be a solution of \eqref{eqn-pma}
such that $u>1$ on $\partial \om \times [-T,0]$ and $u(0,0)=0$.
Let $E$ be an ellipsoid centered at
 the origin such that $|E| \geq 2^{-j} \, |B_1|$ and
$$E  \subset \{ \,  u(x,0) < 1 \, \} \subset \{ \,  u(x,-T) < 1 \, \} \subset B_1(y).
$$
Let $\delta_0,M$ be universal as they appear in Lemma \ref{c1a-l3}
for $d=C_n$ the constant from Lemma \ref{lem-c1a-1}. Then, there
exists a constant $C(j)$ (depending on universal constants and
$j$) such that if $k \ge 0$, $l$ are integers and
$$(C_0^{-k}, (1+\delta_0)^{-l})  \in A_{-t_0} \quad \mbox{for some $t_0\in (0,T]$}$$
and $m_0$ is an integer satisfying   $$3m_0 \le k -l- C(j), \quad
\quad M^{m_0}\, t_0 \geq C(j) \, T,$$ then we can find $\tilde x
\in \{\, u(x,-T) <1 \, \}$ such that
$$u(x,-T) \geq u(\tilde x, - \tilde t) - C_0^{C(j)+l+3m_0-k} + \max_{i=1,2} \{ q_i \cdot (x-\tilde x) \},$$ with
$$ \tilde t = T - \frac T{C(j)} \quad \quad (q_2-q_1)\cdot e_n \geq (1+\delta_0)^{-l-C(j)}.$$

\end{lem}

\begin{rem}\label{c1a-r6} Another way of stating the conclusion of the lemma is that
the translation
\begin{equation}\label{eqn-tilde_u}
\tilde u(x,t):= u(x+\tilde x, t - \tilde t) - u(\tilde x, - \tilde
t), \qquad  \tilde t = T - \frac T{C(j)} \end{equation} satisfies
$$\left ( C_0^{C(j)+l+3m_0-k}, (1+\delta_0)^{-l-C(j)} \right ) \in A_{-T/C(j)}(\tilde u).$$
\end{rem}

\begin{proof} The proof is by induction in $j$.

The case $j=1$ is proved in Lemma \ref{c1a-l3}. Indeed, since
$B_{1/2} \subset E \subset B_1(y) \subset B_{d/2}$ we see that the
hypothesis \eqref{eqn-c1a-6} is satisfied and the conclusion holds
for $\tilde x=0$.

For a general $j$ we start the proof as before. The only
difference here is that we cannot guarantee in the induction step
$m \Rightarrow m+1$ that there exists a section at time
$-t_m=-M^mt_0$ which is $C_n\, d = C_n^2$ balanced around the
origin.

 Let's assume this fails for a first integer $m$.
 By Lemma \ref{lem-c1a-1} we can find a $C_1(j)$-balanced section
(with $C_1(j)>C_n^2$) at the time $-t_m$. The idea is to apply
Proposition \ref{prop-angle-main} as in the induction step and
then to ``replace" the origin with the center of mass $x^*$ of
this section. To be more precise, by Remark \ref{rem-angle5}, the
translation
$$\tilde u(x,t) = u(x^*+x,t-t_m) - u(x^*,-t_m)$$
satisfies
$$\left ( C_2(j)\, C_0^{r_1-k}, (1+\delta_0)^{r_2-l} \, e^{-\delta_0} \right ) \in A_{- \tilde t_0}(\tilde u)$$
with $$\tilde t_0 := c_1(j)\, t_m = c_1(j)\, M^m\, t_0$$ and from
\eqref{eqn-s_m}-\eqref{eqn-c1a-4}
$$r_1 \le 3m +r_2, \quad 0 \le r_2 \le l + C_3(j).$$
Here we assumed that $T>t_m + \tilde t_0$, otherwise the proof is
the same as before by taking $\tilde x=0$, and there is no need to
change the origin. Notice that $m_0 >m$ if $C(j)>1/c_1(j)$.

The above imply
$$(C_0^{-\tilde k}, (1+\delta_0)^{-\tilde l} ) \in A_{- \tilde t_0}(\tilde u)$$
with $$\tilde l := l-r_2 + C_1 \quad  \mbox {and} \quad \tilde k =
k-(3m+r_2) - C_4(j).$$

Now we apply the induction $(j-1)$-step for $\tilde u$. First we
set $$\tilde m_0:=m_0-m \quad \mbox{ and} \quad \tilde T: =
T-t_m,$$ and we have $\tilde T \geq \tilde t_0 \ge c_1(j)t_m \ge
c_2(j)\, T$.

 By Lemma \ref{lem-c1a-1} the maximal ellipsoid centered at the
origin and included in the set $\{\, \tilde u(x,0) < \tilde a\}$
has volume greater than $ 2^{j-1}\, |B_1|$. The constant $\tilde
a=1 - u(x^*,-t_m)$ and by Remark \ref{rem-fs1}, $c_3(j) \le \tilde
a \le 1/c_3(j)$. Thus in order to apply the rescaled induction
step for $\tilde u$ we need to check that (see Remark
\ref{c1a-r4})
$$\tilde k \ge C'(j), \quad 3 \tilde m_0 \le \tilde k -\tilde l
- C'(j), \quad M^{\tilde m_0} \tilde t_0 \ge \tilde TC'(j)$$ for
some large constant $C'(j)$.

If $C(j)$ is sufficiently large then $$M^{\tilde m_0} \, \tilde
t_0 = M^{m_0-m} \, c_1(j) M^m \, t_0 \geq C(j) c_1(j) \, T \geq
C'(j) \, \tilde T,$$ and
\begin{equation}\label{eqn-8.2}
\begin{split}
\tilde k - (\tilde l + 3\tilde m_0) &= k-l - 3(m+\tilde m_0) - C_4(j) - C_1 \\
&=(k-l-3m_0) - C_5(j) \\
&\geq C(j) - C_5(j) \geq C'(j).
\end{split}
\end{equation}
and also,
$$
\tilde k  \geq k - (3m+l) - C_3(j)-C_4(j) $$ $$  \geq
k-(3m_0+l)-C_6(j) \geq C(j)-C_6(j)\ge C'(j).
$$

From the equality in \eqref{eqn-8.2}, $\tilde T \ge c_2(j)T$ and
$\tilde l \le l + C_1$ we clearly obtain the desired result when
we apply the induction step by choosing $C(j)$ sufficiently large.
\end{proof}

\begin{rem}\label{c1a-r7} If in addition to the hypothesis of the lemma we have $$u(0,0) - u(0,-T) \geq a,$$ then
$$u(\tilde x, -\tilde t) - u(\tilde x, -T) \geq \frac a{C(j)} -  \, C_0^{C(j)+l+3m_0-k},
\qquad \mbox{for} \,\, \tilde t = T - \frac T{C(j)}.$$ This and
the conclusion of the lemma imply
$$a \le  C_0^{C(j)+l+3m_0-k}, $$
with $C(j)$ a constant larger than the previous ones.

\end{rem}
\begin{proof}[Proof of Remark \ref{c1a-r7}] From the proof of Lemma \ref{c1a-p5} we see
that when for a certain $m$ we replace $0$ with the center of mass
$x^*$ of the section $S_h:= \{\, u(x, - t_m) \leq l(x) \}$ (for
$l$ linear) with $h=C_0^{r_1-k}$, then
$$u(x^*,-t_m)-l(x^*)\ge -C(j)h.$$
On the other hand, we have
$$u(0,-T) -l(0) \le u(0,-T)-u(0,0) \le - a,$$
and since $u(x,-T)-l(x)$ is negative in $S_h$, at $x^*$ we have
$$u(x^*,-T) -l(x)  \le - \frac a{C(j)}.$$
In conclusion
$$u(x^*,-t_m) \geq u(x^*,-T) + \tilde a, \quad \quad \tilde a:=\frac a{C(j)} - C(j)\, h.$$
Since we perform this change of origin at most $j$ times we obtain
the desired result.
\end{proof}

\begin{proof}[Proof of Theorem \ref{thm-c1a2}] Let
$$(C_0^{-k}, (1+\delta_0)^{-l}) \in A_0, \qquad \mbox{for some}\,\,  k \geq 0$$
where $C_0$ and $\delta_0$ are the constants taken from Lemma
\ref{c1a-p5}. Let $E$ be an ellipsoid of volume $2^{-j}$ around
the origin included in the set $\{ x: u_0(x) < 1 \}$ where $j$
depends on $\dist (k, \partial \, \{ u_0(x) < 1 \}$. In view of
Lemma \ref{c1a-lem7}, it suffices to prove the existence of
constants $\e_0$ and $C_1$ universal and $\tilde C(j)$ such that
\begin{equation}\label{eqn-c1a-10}
l \geq \e_0 \, (k + C_1\, \log T ) - \tilde C(j).
\end{equation}
Since our assumption $(C_0^{-k}, (1+\delta_0)^{-l}) \in A_0$
implies that $l \geq -C_0(j)$ if $k \geq 0$, it follows that
\eqref{eqn-c1a-10} is satisfied,  for some $\tilde C(j)$, if
$$k \leq - C_1\, \log T + C_1(j),$$
where $C_1(j)$ will be specified later. Assume, by contradiction
that \eqref{eqn-c1a-10} does not hold. Thus, since $T\le 1$,
\begin{equation}\label{eqn-c1a-12}
\e_0 \, k > l, \quad \mbox{for some }\,\,\, k > - C_1 \, \log T +
C_1(j)\, \geq C_1(j).
\end{equation}
Using the Lipschitz continuity of $u_0$ we obtain, as in
\eqref{eqn-lip1}, that
$$(C_0^{-k}, (1+\delta_0)^{-l}) \in A_0 \Rightarrow (C_0^{-k+1}, (1+\delta_0)^{-l}
- C(j)\, C_0^{-k}) \in A_{-t_0}$$ with $t_0=c(j)\, C_0^{-k(np+1)}$
which implies that
$$ (C_0^{-k+1}, (1+\delta_0)^{-l-1} ) \in A_{-t_0}.$$
We now apply Lemma \ref{c1a-p5} with $m_0 = [\,  \frac k6\, ]$ and
check that he hypotheses are satisfied. Recall that
$M=C_0^{12(np+1)}$ hence
$$M^{m_0} \, t_0 \geq c(j) \, C_0^{(12 m_0-k)(np+1)} \geq C(j)
\geq C(j)\, T.$$ Also, $l < \e_0 k$ implies that
\begin{equation}\label{eqn-c1a-13}
k \geq \frac {2k}3 \geq 3m_0 + l + C(j)
\end{equation} by choosing  $C_1(j)$ sufficiently large.

Thus, Lemma \ref{c1a-p5} holds. Now we apply the estimate
\eqref{eqn-angle2} for the translation function $\tilde u$ of
\eqref{eqn-tilde_u} that appears in the conclusion of Lemma
\ref{c1a-p5}. In our case
$$\tilde h = C_0^{C(j) + 3m_0 + l -k}, \qquad \tilde \alpha = (1+\delta_0)^{-l - C(j)}, \qquad
\tilde t_0 = \frac T{C(j)}.$$ Since $ S_{\tilde h}' \subset
B_1(y)$ we have $|S_{\tilde h}'|\le C$, hence
$$\tilde h^{1-(n-2) p} \tilde \alpha^{-2p} \geq \frac T{C_2(j)}.$$
Using  \eqref{eqn-c1a-13} we have
$$ (1-(n-2)p)(-\frac k3 ) \, \log C_0  + 2p \, l \,\log (1+\delta_0)\, \geq \log T -  C_3(j)$$
or
$$l - 2\e_0 \, k \geq C\, \log T - C_4(j), \qquad \e_0 := \frac{(1-(n-2)p)\, \log C_0}{12 \, p \,\log (1+\delta_0)}$$
and $C$  universal. We obtain the inequality
$$\e_0 \, k \leq C\, |\log T| + C_4(j)$$
which  contradicts  our assumption \eqref{eqn-c1a-12} if we choose
the constants $C_1$ and $C_1(j)$ appropriately. This concludes the
proof of the theorem.

\end{proof}

We will next sketch the proof of Theorem \ref{thm-c1a3} for the
case $p= \frac 1{n-2}$.

\begin{proof}[Proof of Theorem \ref{thm-c1a3}] The proof is the same as above with
the difference that we need to replace $k$ by $\log k$  in
\eqref{eqn-c1a-10}, i.e. we need to show that there exists $\e_0$
and $C_1$ universal such  that
\begin{equation}\label{eqn-c1a-16}
l \geq \e_0 \, (\log k +  C_1\, \log T ) - \tilde C(j).
\end{equation}

After we apply Lemma \ref{c1a-p5} we know that the translation
$\tilde u$ is a above an angle of opening $\tilde \al$ at time
$-\tilde t_0$ and it separates away from it at most a distance
$\tilde h$ at time $0$. Now we use the stronger estimate
(rescaled) obtained in Proposition \ref{lem-p11} instead of
\eqref{eqn-angle2}. We find
$$ \tilde h \ge c(j)e^{- C \tilde \al^{-1}\tilde
t_0^{-\frac{n-2}{2}}},$$ hence
$$C_0^{C(j) + 3m_0 + l -k} \geq e^{-\frac {C(j)(1+\delta_0)^l}{T^C}}.$$
We obtain
$$\frac k3 \leq  \frac{C(j)\, (1+\delta_0)^l}{T^C},$$
or
$$l \ge 2\e_0 \, \log k + C\, \log T - C(j),$$
and we finish the proof as before.
\end{proof}

We will now proceed with the proof of Theorem \ref{thm-c1a1}.

\begin{proof}[Proof of Theorem \ref{thm-c1a1}]
We begin by observing that since
$u_0(0)=0$, then
$$T \leq C (\,  \| u(\cdot, -T) \|_{L^\infty(\bar S)}).$$
We want to prove that if $(C_0^{-k}, (1+\delta_0)^{-l}) \in A_0$,
for some $k \geq 0$, then
\begin{equation}\label{eqn-c1a-17}
l \geq \e_0 \, k +  C(j, \, a) \quad \mbox{with} \quad a:=\|
\partial_{e_n}u(\cdot, -T) \|_{C^\beta(\bar S)}
\end{equation}
for some $\e_0$ depending on $\beta$ and universal constants. To
show \eqref{eqn-c1a-17} we argue similarly as before. If
\eqref{eqn-c1a-17} doesn't hold, then
$$\e_0\, k >l, \quad \mbox{for some} \,\,\,\,  k > C_1(j, a) .$$
We set $m_0= [ \, \frac k6 \, ] $ and that the hypotheses of Lemma
\ref{c1a-p5} are clearly satisfies. We find that
$$\left ( C_0^{C(j) + 3m_0 + l -k}, (1+\delta_0)^{-l - C(j)} \right ) \in A_{-\tilde T}(\tilde u)$$
from which we conclude that
$$\left ( C_0^{ - \frac k3}, (1+\delta_0)^{-l - C(j)} \right ) \in A_{-\tilde T}(\tilde u).$$
Using that $\partial_{e_n}u(\cdot,-T) \in C^{\beta}$ at $\tilde x$
we obtain
$$\frac{\log (1+\delta_0)}{\log C_0} \, (l+C(j)) \geq \frac {\beta}{\beta+1} \, \frac k3 - C(j,a)$$
from which we derive a contradiction if $\e_0(\beta)$ is chosen
sufficiently small and $ C_1(j,a)$ is chosen large. This concludes
the proof of our theorem.
\end{proof}

We finish with the proof of Theorem \ref{thm-c1a4}.

\begin{proof}[Proof of Theorem \ref{thm-c1a4}] We use the previous
notation. It suffices to show that for some $\e_0$ universal
$$l \ge \e_0 k-C(j,a).$$
From Proposition \ref{prop-p6} we obtain the bound $T \le C(j,a)$.
Now the proof is the same as before. In view of the Remark
\ref{c1a-r7}  our hypothesis implies that
$$C_0^{C(j) + 3m_0 + l -k} \geq a,$$
and the conclusion clearly follows.
\end{proof}

\end{document}